%% file: crtorsionflow28.tex
\documentclass{amsart}

\usepackage{a4wide}
\usepackage{amsmath}
\usepackage{amssymb}
\usepackage{amscd}

\usepackage{enumerate,array,mathrsfs}
\usepackage{graphicx}

\newtheorem{theorem}{Theorem}[section]

\newtheorem{lemma}[theorem]{Lemma}
\newtheorem{proposition}[theorem]{Proposition}
\newtheorem{corollary}[theorem]{Corollary}
\newtheorem*{theorem*}{Theorem}
\theoremstyle{definition}
\newtheorem{definition}[theorem]{Definition}

\newtheorem{example}[theorem]{Example}
\theoremstyle{remark}
\newtheorem{remark}[theorem]{Remark}
\numberwithin{equation}{section}

\theoremstyle{definition}
\theoremstyle{remark}
\numberwithin{equation}{section}

\DeclareMathOperator{\id}{Id}
\DeclareMathOperator{\ad}{ad}
\newcommand{\C}{{\mathbb{C}}}
\newcommand{\R}{{\mathbb{R}}}

\newcommand{\w}{\wedge}
\DeclareMathOperator{\re}{Re}
\DeclareMathOperator{\im}{Im}
\DeclareMathOperator{\Span}{Span}

\begin{document}
\title[Torsion flow]{The torsion flow on a closed pseudohermitian $3$-manifold}
\author{Shu-Cheng Chang$^{ }$}
\address{$^{1}$Department of Mathematics and Taida Institute for
Mathematical Sciences, National Taiwan University, Taipei 10617, Taiwan}
\email{scchang@math.ntu.edu.tw }
\author{Otto van Koert$^{2}$}
\address{$^{2}$Department of Mathematics and Research Institute of
Mathematics, Seoul National University, Building 27, room 402, San 56-1,
Sillim-dong, Gwanak-gu, Seoul, South Korea, Postal code 151-747}
\email{okoert@snu.ac.kr}
\author{Chin-Tung Wu$^{3}$}
\address{$^{3}$Department of Applied Mathematics, National PingTung
University of Education, PingTung 90003, Taiwan}
\email{ctwu@mail.npue.edu.tw }


\begin{abstract}
In this paper we define the torsion flow, a CR analogue of the Ricci flow. 
For homogeneous CR manifolds we give explicit solutions to the torsion flow illustrating various kinds of behavior. 
We also derive monotonicity formulas for CR entropy functionals. 
As an application, we classify torsion breathers.
\end{abstract}

\maketitle

\section{Introduction}

The Ricci flow, introduced by Hamilton, is a geometric flow for metrics on $%
3 $-manifolds, and has played a decisive role in the proof of the
Poincar\'e conjecture and Thurston's geometrization conjecture for $3$%
-manifolds. It is natural to then investigate a corresponding problem for
contact $3$-manifolds. One of way of doing this is to find a CR analogue of
the Ricci flow on a pseudohermitian $3$-manifold (see Section~\ref%
{sec:definitions} for definitions and basic notions in
pseudohermitian geometry).

Recall that a strictly pseudoconvex CR structure on a $3$-manifold $M$ is given by a cooriented
plane field $\ker \theta$, where $\theta$ is a contact form, together with a compatible complex structure $J$.
This gives rise to a natural metric $g=\theta \otimes \theta+d\theta(\cdot
,J \cdot)$ for $M$. Given this data, there is a natural connection, the so-called Tanaka-Webster connection or pseudohermitian connection. We denote the
torsion of this connection by $A_{J,\theta}$, and the Webster curvature, a kind of scalar curvature, by $W$. The torsion flow is then the following PDE, 
\begin{equation}
\left\{ 
\begin{array}{l}
\partial _{t}J_{(t)}=2A_{J_{(t)},\theta _{(t)}}, \\ 
\partial _{t}\theta _{(t)}=-2W\theta _{(t)}.%
\end{array}%
\right.  \label{eq:torsion_flow}
\end{equation}%
It seems to us that the torsion flow \eqref{eq:torsion_flow} is the right CR
analogue of the Ricci flow.

The torsion flow greatly simplifies if the torsion vanishes. This only
happens in very special setups. Indeed, CR $3$-manifolds with vanishing
torsion are $K$-contact, meaning that the Reeb vector field is a Killing
vector field for the metric $g$. In general, one can still hope that the
torsion flow improves properties of the contact manifold underlying the CR-manifold. 

The mostly used tools in the study of Hamilton's Ricci flow \cite{h1}
consist of maximum principles. Exceptions are formed by Hamilton's entropy
formula which holds for closed surfaces with positive Gaussian curvature 
\cite{h2}, and also by Perelman's entropy formulas \cite{pe}. These formulas
can be thought of as monotonicity formulas for integrals of local geometric
quantities.

In this paper, we try to do the same for the torsion flow by setting up some
monotonicity formulas for Perelman-type functionals.
We conclude this introduction with a brief plan of the paper.
\begin{itemize}
\item In Section~\ref{sec:motivation} and Section~\ref{sec:statement_monotonicity_results} we motivate the definition of the torsion flow and give more precise statements of our results.
\item In Section~\ref{sec:definitions} we survey basic notions in CR geometry.
\item In Sections~\ref{section:Tanaka_connection_global_frame} and \ref{section:homogeneous_contact_manifolds} we describe CR manifolds with a global coframe and we also define homogeneous CR manifolds. On the latter class the torsion flow reduces to an ODE if we start with some appropriate initial conditions.
These computations illustrate the behavior of the torsion flow in special cases, and in these cases the torsion
flow behaves as can be expected from a Ricci-like flow.
\item Finally, in Section~\ref{sec:entropy} we discuss analogues of Perelman's entropy formulas.
\end{itemize}

\noindent
{\bf Acknowledgements}
We thank Jih-Hsin Cheng for valuable discussions, and we are in particular grateful to his contributions on the entropy formulas, where he suggested the correct formulation of entropy; he also helped us out with the linearized operator in our attempt to prove short time existence.

The research of SCC and CTW is supported in part by NSC of Taiwan.
OvK is supported by NRF Grant 2012-011755 funded by the Korean government.

\subsection{Motivation for the torsion flow and statement of results}
\label{sec:motivation}
For the basic definitions and notions involved, we refer the reader to
Section~\ref{sec:definitions}. Consider a closed $2n+1$-manifold $M$, with a
smooth family of pseudohermitian structures $(J_{(t)},\theta_{(t)})$ for
which $J_{(t)}$ is compatible with $d\theta_{(t)}$: this means that 
\begin{equation}  \label{eq:herm_metric}
H_{(t)}:=d\theta_{(t)}(\cdot, J_{(t)} \cdot )-id\theta(\cdot,\cdot)
\end{equation}
forms a hermitian metric on the complex vector bundle $(\xi_{(t)}=\ker
\theta_{(t)},J_{(t)})$.

Furthermore, $H_{(t)}$ induces a metric on all tensor fields. We shall use
these metrics and the induced norms without explicitly referring to $H_{(t)}$%
. Throughout the paper, we only consider a fixed contact structure, i.e.~$%
\xi_{(t)}=\ker \theta_{(t)}$ is independent of $t$. Henceforth, we just
write $\xi$.

Take a local orthonormal frame $\{T,Z_{\alpha},Z_{\bar{\beta}}\}$, where $T$ is the
Reeb field, $\{ Z_{\alpha} \}$ is a basis of $(\xi \otimes \mathbb{C})^{1,0}$, and $\{ Z_{\bar{\beta}} \}$ is a basis of $(\xi \otimes \mathbb{C})^{0,1}$. Then we write 
$J=i\theta ^{\alpha}\otimes Z_{\alpha}-i\theta ^{\overline{\alpha}}\otimes Z_{\overline{\alpha}}$%
. Define $E=E_{\alpha}{}^{\bar \beta}\theta ^{\alpha}\otimes Z_{\bar \beta}+E_{\bar \alpha}{}^{\beta}\theta ^{\bar \alpha}\otimes Z_{\beta}$, and consider the general flow on $(M,J,\theta )\times \lbrack 0,T)$ given by 
\begin{equation}
\left\{ 
\begin{array}{l}
\partial _{t}J_{(t)}=2E, \\ 
\partial _{t}\theta _{(t)}=2\eta _{(t)}\theta _{(t)}.%
\end{array}%
\right.  \label{eq:general_torsion_flow}
\end{equation}%
The CR Einstein-Hilbert functional is defined by 
\begin{equation*}
\mathcal{E}(J_{(t)},\theta _{(t)})=\int_{M}Wd\mu .
\end{equation*}%
Here $d\mu =\theta \wedge d\theta^n $ is the volume form and $W$ denotes
Tanaka-Webster curvature. 
From variation formulas for the Webster curvature and the measure, see the appendix and \eqref{18a}, it follows that 
\begin{equation*}
\begin{array}{ccl}
\frac{d}{dt}\mathcal{E}(J_{(t)},\theta _{(t)}) & = & -\int_{M}\{(A^{\bar \alpha}{}_{\beta}E^{\beta}{}_{\bar{\alpha}}
+A^{\alpha}{}_{\bar{\beta}}E^{\bar \beta}_{}{\alpha})-2\eta W\}d\mu \\ 
& = & -2\int_{M}(\Vert A_{J,\theta } \Vert^{2}+W^{2})d\mu \\ 
& \leq & 0%
\end{array}%
\end{equation*}%
if we put $E=A_{J,\theta }$ and $\eta _{(t)}=-W_{(t)}$. 
Here $A_{J,\theta }:=A^{\bar \beta}{}_{\alpha}Z_{\bar \beta}\otimes \theta^{\alpha}+
A^{\beta}{}_{\bar \alpha}Z_{\beta}\otimes \theta^{\bar \alpha}
$ denotes the torsion tensor. 
It is therefore natural to
consider the torsion flow on $M\times \lbrack 0,T)$ as defined in %
\eqref{eq:torsion_flow}. Unfortunately, we do not know whether a
short-time solution to the torsion flow \eqref{eq:torsion_flow} exists in
general.

For the very special class of homogeneous CR manifolds we define in Section~\ref{section:homogeneous_contact_manifolds} we show a short-time existence result.
Furthermore, we show the following convergence result.
\begin{theorem}[Convergence to torsion free CR structure]
Let $(M,\{ \omega^i \}_i,\theta=\omega^1)$ be a homogeneous contact manifold whose Lie algebra is isomorphic to $su(2)$.
Then there is a unique homogeneous complex structure $J_{a_\infty,b=1,c_\infty}$ that is torsion free.
Moreover, for any choice of homogeneous complex structure $J_{a,b=1,c}$, the normalized torsion flow converges to this unique CR-structure $(\ker \theta, J_{a_\infty,b=1,c_\infty})$.

In particular, for any choice of homogeneous complex structure on $SU(2)$, the normalized torsion flow converges to the standard CR-structure.
\end{theorem}

In the next section we shall discuss somewhat technical results concerning
monotonicity properties of Perelman type functionals. 
As an application of these monotonicity results, we classify torsion breathers and solitons. 
The classification of torsion solitons is a necessary step in understanding the singularity formation in the torsion flow. 
Indeed, one expects the
torsion soliton solutions to model finite time singularities of the torsion
flow. In view of the flow \eqref{000a} and original definition in \cite{pe},
it is natural to define the soliton solutions for torsion flow %
\eqref{eq:torsion_flow} as follows.

\begin{definition}
(i) A family $J_{(t)}$ of CR structures on $(M,\theta ,J)$ evolving by the
torsion flow \eqref{eq:torsion_flow} is called a \textbf{breather} if for
some $t_{1}<t_{2}$ and $\delta >0$, there is a contact diffeomorphism $\Phi$ such that
\begin{itemize}
\item $\Phi^*J_{t_1}=J_{t_2}$.
\item $\theta_{(t_{2})}=\lambda \Phi ^{\ast }\theta _{(t_{1})}$.
\end{itemize}
The cases $\lambda =1,\
\lambda <1,\ \lambda >1$ are called  \textbf{steady}, \textbf{shrinking}
or \textbf{expanding} breathers, respectively.

(ii) A breather satisfying the above properties for all pairs of $t_{1}$ and $%
t_{2}$ of real numbers is called a \textbf{torsion soliton}.
\end{definition}

Ideas of Perelman \cite{pe} (see also \cite{ca} and \cite{li}) can be
combined with Theorem \ref{t2}, Theorem \ref{t3} and Theorem \ref{t4} to
show the following classification result.

\begin{corollary}
\label{c51} Let $(M,J,\theta )$ be a closed pseudohermitian $3$-manifold.
Then

(i) there is no closed steady torsion soliton other than the one which
admits zero Tanaka-Webster curvature and vanishing pseudohermitian torsion
up to a contact transformation.

(ii) there is no closed expanding torsion soliton other than the one which
admits negative Tanaka-Webster curvature and vanishing pseudohermitian
torsion.

(iii) there is no closed shrinking torsion soliton other than the one which
admits positive Tanaka-Webster curvature and vanishing pseudohermitian
torsion.
\end{corollary}

\subsection{Some monotonicity results for Perelman-type functionals}

\label{sec:statement_monotonicity_results} The statements in this section
are of a more technical nature: we will derive the CR analogue of Perelman's
monotonicity formulas for the so-called coupled torsion flows \eqref{eq:coupled_torsion_flow}, \eqref{2008-3} and \eqref{2008-2} in dimension $3$.

In Section \ref{sec:F_functional} we define the CR analogue of Perelman's $%
\mathcal{F}$-functional by 
\begin{equation*}
\begin{array}{c}
\mathcal{F}(J_{(t)},\theta _{(t)},\varphi _{(t)})=\int_{M}(W+\left \vert
\nabla _{b}\varphi \right \vert _{J,\theta }^{2})e^{-\varphi }d\mu%
\end{array}%
\end{equation*}%
with the constraint 
\begin{equation*}
\begin{array}{c}
\int_{M}e^{-\varphi }d\mu =1.
\end{array}%
\end{equation*}%
Under the flow \eqref{eq:general_torsion_flow}, this is equivalent to  
\begin{equation*}
\begin{array}{c}
\int_{M}(\varphi _{t}-4\eta _{(t)})e^{-\varphi }d\mu =0.
\end{array}%
\end{equation*}%
Therefore, the following coupled torsion flow is natural,
\begin{equation}
\left \{ 
\begin{array}{l}
\partial _{t}J_{(t)}=2E, \\ 
\partial _{t}\theta _{(t)}=2\eta _{(t)}\theta _{(t)}, \\ 
\partial _{t}\varphi _{(t)}=4\eta _{(t)},%
\end{array}%
\right.  \label{eq:coupled_torsion_flow}
\end{equation}%
with $E_{11}=e^{\varphi }(A_{11}-i\varphi _{11}-i\varphi _{1}\varphi _{1})\ $%
and $\eta _{(t)}=e^{\varphi }(2{\Delta _{b}}\varphi -\left \vert \nabla
_{b}\varphi \right \vert _{J,\theta }^{2}+W).$

\begin{theorem}
\label{t2} Let $(M,J,\theta )$ be a closed $3$-dimensional pseudohermitian
manifold and $J_{(t)},\theta _{(t)},\varphi _{(t)}$ be a solution of the
coupled torsion flow \eqref{eq:coupled_torsion_flow} on $M\times \lbrack 0,T).$ Then 
\begin{equation*}
\begin{array}{ccl}
\text{{\small $\frac{d}{dt}\mathcal{F}(J_{(t)},\theta _{(t)},\varphi _{(t)})$%
}} & = & -2\int_{M}(2{\Delta _{b}}\varphi -\left \vert \nabla _{b}\varphi
\right \vert _{J,\theta }^{2}+W)^{2}d\mu \\ 
&  & -2\int_{M}|A_{11}-i\varphi _{11}-i\varphi _{1}\varphi _{1}|^{2}d\mu \\ 
& \leq & 0.%
\end{array}%
\end{equation*}%
The monotonicity formula is strict unless%
\begin{equation*}
A_{11}-i\varphi _{11}-i\varphi _{1}\varphi _{1}=0\ \ \ \ \ \mathrm{and\ }\ \
\ 2{\Delta _{b}}\varphi -\left \vert \nabla _{b}\varphi \right \vert
_{J,\theta }^{2}+W=0.
\end{equation*}%
That is, up to a contact transformation $\widetilde{\theta }=e^{-\varphi
}\theta $%
\begin{equation*}
\text{\ }\widetilde{A_{11}}=0\ \ \ \ \mathrm{and}\ \ \ \widetilde{W}=0.
\end{equation*}
\end{theorem}

\begin{remark}
Observe that for $\widetilde{\theta }=e^{-\varphi }\theta $, 
\begin{equation}
\left \{ 
\begin{array}{ccc}
\widetilde{A_{11}} & = & e^{\varphi }(A_{11}-i\varphi _{11}-i\varphi
_{1}\varphi _{1}), \\ 
\widetilde{W} & = & e^{\varphi }(2{\Delta _{b}}\varphi -\left \vert \nabla
_{b}\varphi \right \vert _{J,\theta }^{2}+W).%
\end{array}%
\right.  \label{2008}
\end{equation}%
Then the coupled torsion flow \eqref{eq:coupled_torsion_flow} on $(M,J_{(t)},\theta _{(t)})$
is equivalent to the following system of coupled torsion flows on $%
(M,J_{(t)},\widetilde{\theta }_{(t)})$ 
\begin{equation*}
\mathcal{F}(\widetilde{J}_{(t)},\widetilde{\theta }_{(t)},\varphi _{(t)})=%
\mathcal{E}(\widetilde{J}_{(t)},\widetilde{\theta }_{(t)})=\int_{M}%
\widetilde{W}d\widetilde{\mu }
\end{equation*}%
and 
\begin{equation*}
\left \{ 
\begin{array}{l}
\partial _{t}\widetilde{J}_{(t)}=2A_{\widetilde{J}_{(t)},\widetilde{\theta }%
_{(t)}}, \\ 
\partial _{t}\widetilde{\theta }_{(t)}=-2\widetilde{W}_{(t)}\widetilde{%
\theta }_{(t)}, \\ 
\partial _{t}\varphi _{(t)}=4\widetilde{W}_{(t)}.%
\end{array}%
\right.
\end{equation*}
\end{remark}

In Section~\ref{sec:Wfunctional} we define two functionals analogous to
Perelman's $\mathcal{W}$-functional, namely the $\mathcal{W}^{+}$%
-functional, 
\begin{equation*}
\begin{array}{c}
\mathcal{W}^{+}(J_{(t)},\theta _{(t)},\varphi _{(t)},\tau
_{(t)})=\int_{M}[\tau (W+\left \vert \nabla _{b}\varphi \right \vert
_{J,\theta }^{2})+\frac{1}{2}\varphi -1](4\pi \tau )^{-2}e^{-\varphi }d\mu%
\end{array}%
\end{equation*}%
and the $\mathcal{W}^{-}$-functional 
\begin{equation*}
\begin{array}{c}
\mathcal{W}^{-}(J_{(t)},\theta _{(t)},\varphi _{(t)},\tau
_{(t)})=\int_{M}[\tau (W+\left \vert \nabla _{b}\varphi \right \vert
_{J,\theta }^{2})-\frac{1}{2}\varphi +1](4\pi \tau )^{-2}e^{-\varphi }d\mu .%
\end{array}%
\end{equation*}

\begin{remark}
Note that $\mathcal{W}^{+}$ and $\mathcal{W}^{-}$ are invariant under the
rescaling $\tau \longmapsto c\tau $ and $\theta \longmapsto c\theta$.
Furthermore, we have $\mathcal{W}^{\pm }(J,\theta ,\varphi ,\tau )=\mathcal{W%
}^{\pm }(\Phi ^{\ast}J,\Phi ^{\ast }\theta ,\varphi \circ \Phi ,\tau )$ for
a contact diffeomorphism $\Phi :M\rightarrow M.$
\end{remark}

In view of Theorem \ref{t2}, we first study the monotonicity property of $%
\mathcal{W}^{+}$-functional. By the same discussion as before, the
constraint 
\begin{equation*}
\begin{array}{c}
\int_{M}(4\pi \tau )^{-2}e^{-\varphi }d\mu =1%
\end{array}%
\end{equation*}%
is equivalent to another constraint, namely
\begin{equation*}
\begin{array}{c}
\int_{M}(\varphi _{t}+2\tau ^{-1}\frac{d\tau }{dt}-4\eta _{(t)})(4\pi \tau
)^{-2}e^{-\varphi }d\mu =0%
\end{array}%
\end{equation*}%
under the flow \eqref{eq:general_torsion_flow}. Therefore we consider the
following coupled torsion flow: 
\begin{equation}
\left \{ 
\begin{array}{l}
\partial _{t}J_{(t)}=2E, \\ 
\partial _{t}\theta _{(t)}=2\eta _{(t)}\theta _{(t)}, \\ 
\partial _{t}\varphi _{(t)}=4(\eta _{(t)}-\tau ^{-1}), \\ 
\partial _{t}\tau =2,%
\end{array}%
\right.  \label{2008-3}
\end{equation}%
with $E_{11}=(A_{11}-i\varphi _{11}-i\varphi _{1}\varphi _{1})\ $and $\eta
_{(t)}=(2{\Delta _{b}}\varphi -\left \vert \nabla _{b}\varphi \right \vert
_{J,\theta }^{2}+W).$

\begin{theorem}
\label{t3} Let $(M,J,\theta )$ be a closed $3$-dimensional pseudohermitian
manifold and $J_{(t)},\theta _{(t)},\varphi _{(t)}$ and $\tau _{(t)}$ be a
solution of the coupled torsion flow (\ref{2008-3}). Then 
\begin{equation*}
\begin{array}{l}
\text{{\small $\frac{d}{dt}\mathcal{W}^{+}(J_{(t)},\theta _{(t)},\varphi
_{(t)},\tau _{(t)})$}} \\ 
=-2\tau \int_{M}(2{\Delta _{b}}\varphi -\left \vert \nabla _{b}\varphi
\right \vert _{J,\theta }^{2}+W-\tau ^{-1})^{2}(4\pi \tau )^{-2}e^{-\varphi
}d\mu \\ 
\ \ -2\tau \int_{M}|A_{11}-i\varphi _{11}-i\varphi _{1}\varphi
_{1}|^{2}(4\pi \tau )^{-2}e^{-\varphi }d\mu \\ 
\leq 0.%
\end{array}%
\end{equation*}%
The monotonicity formula is strict unless%
\begin{equation*}
A_{11}-i\varphi _{11}-i\varphi _{1}\varphi _{1}=0\ \ \ \ \ \mathrm{and\ }\ \
\ 2{\Delta _{b}}\varphi -\left \vert \nabla _{b}\varphi \right \vert
_{J,\theta }^{2}+W-\tau ^{-1}=0.
\end{equation*}%
That is, up to a contact transformation $\widetilde{\theta }=e^{-\varphi
}\theta $%
\begin{equation*}
\widetilde{A_{11}}=0\text{\ \ \textrm{and\ }\ }\widetilde{W}-\tau
^{-1}e^{\varphi }=0.
\end{equation*}
\end{theorem}

Next we study the monotonicity property of $\mathcal{W}^{-}$-functional 
\begin{equation*}
\begin{array}{c}
\mathcal{W}^{-}(J_{(t)},\theta _{(t)},\varphi _{(t)},\tau
_{(t)})=\int_{M}[\tau (W+\left \vert \nabla _{b}\varphi \right \vert
_{J,\theta }^{2})-\frac{1}{2}\varphi +1](4\pi \tau )^{-2}e^{-\varphi }d\mu .%
\end{array}%
\end{equation*}%
By the same discussion as before, the constraint 
\begin{equation*}
\begin{array}{c}
\int_{M}(4\pi \tau )^{-2}e^{-\varphi }d\mu =1%
\end{array}%
\end{equation*}%
is equivalent to
\begin{equation*}
\begin{array}{c}
\int_{M}(\varphi _{t}+2\tau ^{-1}\frac{d\tau }{dt}-4\eta _{(t)})(4\pi \tau
)^{-2}e^{-\varphi }d\mu =0%
\end{array}%
\end{equation*}%
under the flow \eqref{eq:general_torsion_flow}. Therefore we consider the
following coupled torsion flow : 
\begin{equation}
\left \{ 
\begin{array}{l}
\partial _{t}J_{(t)}=2E, \\ 
\partial _{t}\theta _{(t)}=2\eta _{(t)}\theta _{(t)}, \\ 
\partial _{t}\varphi _{(t)}=4(\eta _{(t)}+\tau ^{-1}), \\ 
\partial _{t}\tau =-2,%
\end{array}%
\right.  \label{2008-2}
\end{equation}%
with $E_{11}=(A_{11}-i\varphi _{11}-i\varphi _{1}\varphi _{1})\ $and $\eta
_{(t)}=(2{\Delta _{b}}\varphi -\left \vert \nabla _{b}\varphi \right \vert
_{J,\theta }^{2}+W).$

\begin{theorem}
\label{t4} Let $(M,J,\theta )$ be a closed $3$-dimensional pseudohermitian
manifold and $J_{(t)},\theta _{(t)},\varphi _{(t)}$ and $\tau _{(t)}$ be a
solution of the coupled torsion flow (\ref{2008-2}). Then 
\begin{equation*}
\begin{array}{l}
\text{{\small $\frac{d}{dt}\mathcal{W}^{-}(J_{(t)},\theta _{(t)},\varphi
_{(t)},\tau _{(t)})$}} \\ 
=-2\tau \int_{M}(2{\Delta _{b}}\varphi -\left \vert \nabla _{b}\varphi
\right \vert _{J,\theta }^{2}+W+\tau ^{-1})^{2}(4\pi \tau )^{-2}e^{-\varphi
}d\mu \\ 
\ \ \ -2\tau \int_{M}|A_{11}-i\varphi _{11}-i\varphi _{1}\varphi
_{1}|^{2}(4\pi \tau )^{-2}e^{-\varphi }d\mu \\ 
\leq 0.%
\end{array}%
\end{equation*}%
The monotonicity formula is strict unless%
\begin{equation*}
A_{11}-i\varphi _{11}-i\varphi _{1}\varphi _{1}=0\ \ \ \ \ \mathrm{and\ }\ \
\ 2{\Delta _{b}}\varphi -\left \vert \nabla _{b}\varphi \right \vert
_{J,\theta }^{2}+W+\tau ^{-1}=0.
\end{equation*}%
That is, up to a contact transformation $\widetilde{\theta }=e^{-\varphi
}\theta $%
\begin{equation*}
\widetilde{A_{11}}=0\text{\ \ \textrm{and\ }\ }\widetilde{W}+\tau
^{-1}e^{\varphi }=0.
\end{equation*}
\end{theorem}

\begin{remark}
Note that for $\widetilde{\theta }_{(t)}=e^{-\varphi }\theta _{(t)},$ we may
reparametrize the time $t$ by the formula $\widetilde{t}=\int_{0}^{t}e^{-\varphi
_{(s)}(x(s))}ds$.
We have $\frac{d\widetilde{t}}{dt}=e^{-\varphi
_{(t)}(x(t))}$, so the coupled torsion flows \eqref{2008-3} and \eqref
{2008-2} on $(M,J_{(t)},\theta _{(t)})$ are equivalent to the following
coupled torsion flows on $(M,\widetilde{J}_{(\widetilde{t})},\widetilde{%
\theta }_{(\widetilde{t})}),$ respectively : 
\begin{equation*}
\left \{ 
\begin{array}{l}
\partial _{\widetilde{t}}\widetilde{J}_{(\widetilde{t})}=2A_{\widetilde{J}_{(%
\widetilde{t})},\widetilde{\theta }_{(\widetilde{t})}}, \\ 
\partial _{\widetilde{t}}\widetilde{\theta }_{(\widetilde{t})}=-2\widetilde{W%
}\widetilde{\theta }_{(\widetilde{t})}, \\ 
\partial _{\widetilde{t}}\varphi _{(\widetilde{t})}=4(\widetilde{W}-\tau
^{-1}e^{\varphi }), \\ 
\partial _{t}\tau =2,%
\end{array}%
\right.
\end{equation*}%
and%
\begin{equation*}
\left \{ 
\begin{array}{l}
\partial _{\widetilde{t}}\widetilde{J}_{(\widetilde{t})}=2A_{\widetilde{J}_{(%
\widetilde{t})},\widetilde{\theta }_{(\widetilde{t})}}, \\ 
\partial _{\widetilde{t}}\widetilde{\theta }_{(\widetilde{t})}=-2\widetilde{W%
}\widetilde{\theta }_{(\widetilde{t})}, \\ 
\partial _{\widetilde{t}}\varphi _{(\widetilde{t})}=4(\widetilde{W}+\tau
^{-1}e^{\varphi }), \\ 
\partial _{t}\tau =-2.%
\end{array}%
\right.
\end{equation*}
\end{remark}

Recall that $X_{f}$ is called a contact vector field if the Lie derivative $\mathcal{L}_{X_{f}}\theta
=\eta \theta $ for some function $\eta$. 
Such a contact vector field has the form $X_{f}=if_{1}Z_{%
\overline{1}}-if_{\overline{1}}Z_{1}-fT$\ \ for some smooth function $f$ on $%
M$. 
Furthermore 
\begin{equation}
\mathcal{L}_{X_{f}}J=2B_{J}^{\prime }f\ :=(f_{11}+iA_{11}f)\theta
^{1}\otimes Z_{\overline{1}}+(f_{\overline{1}\overline{1}}-iA_{\overline{1}%
\overline{1}}f)\theta ^{\overline{1}}\otimes Z_{1}  \label{000b}
\end{equation}%
so \eqref{eq:coupled_torsion_flow} is equivalent to 
\begin{equation}
\left \{ 
\begin{array}{l}
\partial _{t}J_{(t)}=2JB_{J}^{\prime }f=J\mathcal{L}_{X_{f}}J, \\ 
\partial _{t}\theta _{(t)}=2\eta _{(t)}\theta _{(t)}, \\ 
\partial _{t}\varphi _{(t)}=4\eta _{(t)}.%
\end{array}%
\right.  \label{000a}
\end{equation}%
with $f=e^{\varphi }$.
Similar results hold for \eqref{2008-3} and \eqref{2008-2}. 

\section{Preliminaries and definitions}

\label{sec:definitions} 
In this section we introduce some basic notions from
pseudohermitian geometry. We learned many of these notions from \cite{l1,l2}, and we refer to these papers for proofs and more references.

\begin{definition}
\label{def:CR}
Let $M$ be a smooth manifold and $\xi\subset TM$ a subbundle.
A {\bf CR structure} on $\xi$ consists of an endomorphism $J:\xi \to \xi$ with $J^2=-\id$ such that the following integrability condition holds.
\begin{enumerate}
\item if $X,Y\in \xi$, then so is $[JX,Y]+[X,JY]$.
\item $J([JX,Y]+[X,JY])=[JX,JY]-[X,Y]$.
\end{enumerate}
\end{definition}
The CR structure $J$ can be extended to $\xi \otimes \C$,
which we can then decompose into the direct sum of eigenspaces of $J$. The
eigenvalues of $J$ are $i$ and $-i$, and the corresponding eigenspaces will
be denoted by $T^{1,0}$ and $T^{0,1}$, respectively.
The integrability condition can then be reformulated as 
$$
X,Y\in T^{1,0} \text{ implies } [X,Y]\in T^{1,0}.
$$

Now consider a closed $2n+1$-manifold $M$ with a cooriented contact structure $\xi=\ker \theta $. This means that $\theta \wedge d\theta^n \neq 0$.
The \textbf{Reeb vector field} of $\theta$ is the vector field $T$
uniquely determined by the equations 
\begin{equation}
{\theta }(T)=1,\quad \text{and}\quad d{\theta }(T,{\cdot }%
)=0.  
\label{eq:defining_Reeb}
\end{equation}%
\begin{definition}
A \textbf{pseudohermitian manifold} is a triple $(M^{2n+1},\theta,J)$ where
\begin{itemize}
\item $\theta$ is a contact form on $M$.
\item $J$ is a CR structure on $\ker \theta$.
\end{itemize}
\end{definition}

\begin{definition}
The \textbf{Levi form} $\left \langle \ ,\ \right \rangle $ is the Hermitian
form on $T^{1,0}$ defined by 
\begin{equation*}
H(Z,W)=\left \langle Z,W\right \rangle =-i\left \langle d\theta ,Z\wedge 
\overline{W}\right \rangle .
\end{equation*}
\end{definition}
We can extend this Hermitian form $\left \langle \ ,\ \right \rangle $ to $T^{0,1}$ by defining $%
\left \langle \overline{Z},\overline{W}\right \rangle =\overline{%
\left
\langle Z,W\right \rangle }$ for all $Z,W\in T^{1,0}$. 
Furthermore, the Levi form
naturally induces a Hermitian form on the dual bundle of $T^{1,0}$, and
hence on all induced tensor bundles.

We now restrict ourselves to strictly pseudoconvex manifolds, or in other words to compatible complex structures $J$. 
This means that the Levi form induces a Hermitian metric $\langle \cdot ,\cdot \rangle _{J,{\theta }}$ by 
\begin{equation*}
\langle V,U\rangle _{J,{\theta }}=d{\theta }(V,JU).
\end{equation*}
The associated norm is defined as usual: $%
|V|_{J,\theta }^{2}=\langle V,V\rangle _{J,{\theta }}$. 
It follows that $H$ also gives rise to a Hermitian metric for $T^{1,0}$%
, and hence we obtain Hermitian metrics on all induced tensor bundles.
By integrating this Hermitian metric over $M$ with respect to the volume
form $d\mu =\theta \wedge d\theta^n $, we get an $L^2$-inner product on the
space of sections of each tensor bundle.
\begin{definition}
Let $(M,J,\xi =\ker \theta _{0})$ be a CR $3$-manifold
that is the smooth boundary of a bounded, strictly
pseudoconvex domain in a complete Stein manifold $V^{4}$. 
We shall call such a CR $3$-manifold {\bf Stein fillable}.
\end{definition}

\subsection{Pseudohermitian connection}
\label{sec:pseudoherm_connection}
Consider a local frame $\left\{ T,Z_{\alpha},Z_{\bar{\beta}}\right\} $ for $TM\otimes \mathbb{C}$, where $\{ Z_{\alpha} \}$ is a local
frame for $T^{1,0}$, and $Z_{\bar{\beta}}=\overline{Z_{\beta}}$ a local frame for $T^{0,1}$. 
Then $%
\left\{ \theta ,\theta ^{\alpha},\theta ^{\bar{\beta}}\right\} $, the coframe dual to 
$\left\{ T,Z_{\alpha},Z_{\bar{\beta}}\right\} $, satisfies 
\begin{equation*}
d\theta =ih_{\alpha\bar{\beta}}\theta ^{\alpha}\wedge \theta ^{\bar{\beta}},
\end{equation*}
where $h_{\alpha \bar{\beta}}$ is a positive definite matrix. 
By the Gram-Schmidt process we can always choose $Z_{\alpha}$ such that $h_{\alpha \bar \beta}=\delta_{\alpha\bar \beta}$; throughout this paper, we shall take such a local frame.

The {\bf pseudohermitian connection} or {\bf Tanaka-Webster connection} of $(J,\theta )$ is the connection $\nabla $
on $TM\otimes \mathbb{C}$ (and extended to tensors) given in terms of a
local frame $\{Z_{\alpha} \}$ for $T^{1,0}$ by

\begin{equation*}
\nabla Z_{\alpha}=\omega _{\alpha}{}^{\beta}\otimes Z_{\beta},
\quad \nabla Z_{\bar{ \alpha}}=\omega
_{\bar{\alpha}}{}^{\bar{\beta}}\otimes Z_{\bar{\beta}},\quad \nabla T=0,
\end{equation*}%
where $\omega _{\alpha}{}^{\beta}$ is the $1$-form uniquely determined by the
following equations:

\begin{equation}  
\label{eq:defining_Tanaka_Webster}
\begin{split}
d\theta ^{\beta}& =\theta ^{\alpha}\wedge \omega _{\alpha}{}^{\beta}+\theta \wedge \tau ^{\beta} \\
\tau_\alpha \w \theta^{\alpha} & =0 \\
\omega _{\alpha}{}^{\beta}+\omega _{\bar{\beta}}{}^{\bar{\alpha}} &=0.
\end{split}%
\end{equation}%
Here $\tau ^{\alpha}$ is called the \textbf{pseudohermitian torsion}, which we can also write as
$$
\tau_\alpha=A_{\alpha\beta}\theta ^{\beta}.
$$
The components $A_{\alpha \beta}$ satisfy
$$
A_{\alpha \beta}=A_{\beta \alpha}.
$$
We often consider the \textbf{%
torsion tensor} given by 
\begin{equation*}
A_{J,\theta}=A^{\alpha}{}_{\bar{\beta}}Z_{\alpha} \otimes \theta^{\bar \beta} +A^{\bar
\alpha}{}_{\beta}Z_{\bar \alpha} \otimes \theta^{\beta}.
\end{equation*}

The following remark gives some geometric meaning to the
pseudohermitian torsion.

\begin{remark}
Let $X_{f}$ be the contact vector field for a real-valued function $f\in
C^{2}(M)$. Then we have 
\begin{equation*}
\mathcal L_{X_{f}}\theta =-(Tf)\theta,
\end{equation*}%
and
\begin{equation*}
L_{X_{f}}J =2B_{J}^{\prime }f:=
2(f^{\bar \beta}{}_{\alpha}+iA_{\alpha}{}^{\bar \beta})\theta^\alpha \otimes Z_{\bar \beta}+
2(f^{\beta}{}_{\bar \alpha}-iA_{\bar \alpha}{}^{\beta})\theta^{\bar \alpha} \otimes Z_{\beta}
\end{equation*}%
See for instance \cite[Lemma 3.4]{cl1}. 
In particular, we have $X_f=T$ for $f=1$, and the above equation reduces to 
\begin{equation*}
\mathcal{L}_T J=2JA_{J,\theta},
\end{equation*}
so we see that the torsion tensor measures to what extend the complex
structure $J$ is invariant under the Reeb flow.
\end{remark}

We now consider the curvature of the Tanaka-Webster connection in terms of the coframe $\{\theta=\theta^0,\theta^\alpha,\theta^{\bar \beta} \}$.
The second structure equation gives
\[
\begin{split}
\Omega_{\beta }{}^{\alpha }& =\overline{\Omega _{\bar{\beta}}{}^{\bar{\alpha}}}%
=d\omega _{\beta }{}^{\alpha }-\omega _{\beta }{}^{\gamma }\wedge \omega
_{\gamma }{}^{\alpha }, \\
\Omega _{0}{}^{\alpha }& =\Omega _{\alpha }{}^{0}=\Omega _{0}{}^{\bar{\beta}}=\Omega _{%
\bar{\beta}}{}^{0}=\Omega _{0}{}^{0}=0.
\end{split}
\]

In \cite[Formulas 1.33 and 1.35]{we}, Webster showed that the curvature $\Omega _{\beta }{}^{\alpha }$ can be written as
\begin{equation}
\begin{array}{c}
\Omega _{\beta }{}^{\alpha }=R_{\beta }{}^{\alpha }{}_{\rho \bar{\sigma}}\theta
^{\rho }\wedge \theta ^{\bar{\sigma}}+W_{\beta }{}^{\alpha }{}_{\rho }\theta
^{\rho }\wedge \theta -W^{\alpha }{}_{\beta \bar{\rho}}\theta ^{\bar{\rho}%
}\wedge \theta +i\theta _{\beta }\wedge \tau ^{\alpha }-i\tau _{\beta
}\wedge \theta ^{\alpha },%
\end{array}
\label{eq:decomposing_curvature}
\end{equation}%
where the coefficients satisfy 
\begin{equation*}
\begin{array}{c}
R_{\beta \bar{\alpha}\rho \bar{\sigma}}=\overline{R_{\alpha \bar{\beta}%
\sigma \bar{\rho}}}=R_{\bar{\alpha}\beta \bar{\sigma}\rho }=R_{\rho \bar{%
\alpha}\beta \bar{\sigma}},\ \ W_{\beta \bar{\alpha}\gamma }=W_{\gamma \bar{%
\alpha}\beta }.%
\end{array}%
\end{equation*}%
In addition, by \cite[(2.4)]{l2} the coefficients $W_\alpha{}^\beta{}_\rho$ are determined by the torsion,
$$
W_\alpha{}^\beta{}_\rho=A_{\alpha \rho,}{}^\beta.
$$
Contraction of \eqref{eq:decomposing_curvature} yields%
\begin{equation}
\begin{split}
\Omega _{\alpha }{}^{\alpha }=d\omega _{\alpha }{}^{\alpha }
&=R_{\rho \bar{\sigma%
}}\theta ^{\rho }\wedge \theta ^{\bar{\sigma}}+W_{\alpha }{}^{\alpha
}{}_{\rho }\theta ^{\rho }\wedge \theta -W_{\overline{\alpha }}{}^{\overline{%
\alpha }}{}_{\bar{\rho}}\theta ^{\bar{\rho}}\wedge \theta \\
&=R_{\rho \bar{\sigma%
}}\theta ^{\rho }\wedge \theta ^{\bar{\sigma}}
+A_{\alpha \rho}{}^\alpha \theta^\rho \w \theta
-A_{\bar \alpha \bar \rho}{}^{\bar \alpha} \theta^{\bar \rho} \w \theta .
\end{split}
\label{eq:curvature_Tanaka_connection}
\end{equation}

\begin{definition}
The {\bf pseudohermitian Ricci tensor} is the tensor with components $R_{\rho \bar \sigma}$.
Its trace $W:=R_{\rho}{}^\rho$ is called the {\bf Webster curvature}.
If the pseudohermitian Ricci tensor is a scalar multiple of the Levi form, then we say that the pseudohermitian structure is {\bf pseudo-Einstein}.
\end{definition}

\begin{remark}
From the definition it is clear that the Webster curvature is the analogue of the scalar curvature in Riemannian geometry, and we also see that the pseudo-Einstein condition mimics the Einstein condition.
Unlike the Riemannian case, pseudo-Einstein structures do not necessarily have constant Webster curvature, even in higher dimensions.
\end{remark}

We will denote components of covariant derivatives by indices preceded by a
comma.
For instance, we write $A_{\alpha\beta,\gamma}$.
Here the indices $\{0,\alpha,\bar{\beta}\}$ indicate derivatives with respect to $\{T,Z_{\alpha},Z_{%
\bar{\beta}}\}$. 
For derivatives of a scalar function, we will often omit the
comma.
For example, $\varphi _{\alpha}=Z_{\alpha}\varphi ,\ \varphi _{\alpha\bar{\beta}}=Z_{%
\bar{\beta}}Z_{\alpha}\varphi -\omega _{\alpha}{}^{\gamma}(Z_{\bar{\beta}})Z_{\gamma}\varphi ,\ \varphi
_{0}=T \varphi $ for a (smooth) function $\varphi$.

\subsection{Discussion on short time existence}
\label{sec:short_time}
We have no short time existence result for general initial conditions, but we shall discuss existence and properties of the torsion flow for homogeneous contact manifolds in Sections~\ref{section:Tanaka_connection_global_frame} and \ref{section:homogeneous_contact_manifolds}.
Here we thank Jih-Hsin Cheng for valuable contributions, in particular involving the variation formulas and linearized operator. See also his paper for a related flow, \cite{c}.

In general, it is possible to show the following.

\begin{theorem}
Suppose for initial condition $(J_0,\theta_0)$ the torsion flow preserves the condition
\begin{equation}
\label{eq:pseudo_Einstein_condition}
W_{1}-iA^{\bar 1}_{\phantom{1}1,\overline{1}}=0.
\end{equation}
Then there are $\epsilon>0$, and smooth family of smooth tensors $(J_t,\theta_t)$ on $t\in[0,\epsilon[$ such that $(J_t,\theta_t)$ is a solution to the torsion flow.
In addition, such a solution is unique, and the curvature evolution satisfies
\[
\begin{split}
\dot A_{\bar \alpha \bar \gamma} &=2iW_{\bar \alpha \bar \gamma}+2WA_{\bar \alpha \bar \gamma}-iA_{\bar \alpha \bar \gamma,0} \\
\dot W &=4 \Delta _{b}W+2( W^2- \Vert A \Vert^2 )+2 \Re(iA_{\gamma \alpha}{}^{,\gamma \alpha}.
\end{split}
\]
\end{theorem}

\begin{remark}
If Condition~\eqref{eq:pseudo_Einstein_condition} holds on the entire interval $[0,\epsilon[$, the torsion flow is of heat type on the interval  $[0,\epsilon[$. However this  Condition~\eqref{eq:pseudo_Einstein_condition} is not preserved in general, although it is for the special class of homogeneous contact manifolds. 
\end{remark}

\begin{proof}[Sketch of the proof]
This assertion can be proved by carrying out the following steps.
Linearize the torsion flow. If the Condition~\eqref{eq:pseudo_Einstein_condition} holds on a time interval $[0,\epsilon]$, commutation relations, \cite[Equation 2.15]{l2}, can be used to show that the torsion flow is of heat type.
The arguments of Hamilton, \cite[Section 4,5,6]{h1}, can then be used to show short-time existence.
The main ingredient is the Nash-Moser inverse function theorem.

The curvature evolution equations can be found by working out the variation formulas for the the torsion and Webster curvature.
This is done in the appendix.
\end{proof}
We omit the details since Condition~\eqref{eq:pseudo_Einstein_condition} is in practice not preserved and even in the special cases it is, checking this involves solving the torsion flow.

\section{Contact $3$-manifolds with a global frame and pseudohermitian
structures}

\label{section:Tanaka_connection_global_frame}
We now specialize to $3$-dimensional pseudohermitian manifolds.
Let $(M,\xi=\ker \theta)$ be
a cooriented contact $3$-manifold. 
Denote the Reeb field by $T$. 
Furthermore, in this section and in the next, Section~\ref{section:homogeneous_contact_manifolds}, we shall assume that $\xi$ admits a global symplectic trivialization, i.e.~there
are vector fields $U,V$ such that $\xi =\Span(U,V)$ and $d\theta(U,V)=1$.

\begin{lemma}
\label{lemma:global_frame_chern_class} Let $(M,\xi=\ker \theta)$ be a
contact $3$-manifold. Then there is a global trivialization $U,V$ of its
contact structure if and only if $c_1(\xi)=0$.
\end{lemma}

\begin{proof}
The contact structure $\xi$ admits the structure of a symplectic vector
bundle $(\xi,d\theta)$. By choosing a compatible complex structure $J$, we
obtain a complex line bundle $(\xi,J)$. It is well-known that smooth complex
line bundles are trivial if and only if their first Chern class vanishes,
see \cite[Chapter III, Section 4]{Wells}.
\end{proof}

The last step needs the first Chern class with integer coefficients.
Chern-Weil theory will not suffice in general. Henceforth, we shall assume
that the globally defined vector fields $U,V$ form a symplectic basis of $(\xi ,d\theta )$. Consider the
coframe $\theta ,\alpha ,\beta $ dual to $T,U,V$. Then $d\theta (U,V)=1$, so 
\begin{equation}
d\theta =\alpha {\wedge }\beta .  \label{eq:normalized_frame}
\end{equation}

\begin{lemma}
\label{lemma:cpx_structure_matrix} Let $J$ be a compatible complex structure
for the symplectic vector bundle $(\xi =\ker \theta ,d\theta )$. Then there
are smooth functions $a:M\rightarrow \R$ and $c:M\rightarrow \R_{>0}$ such that, with respect to the frame $U,V$, the complex structure $J
$ is represented by the matrix 
\begin{equation*}
J=\left( 
\begin{array}{cc}
a & -\frac{1+a^{2}}{c} \\ 
c & -a%
\end{array}%
\right) .
\end{equation*}
\end{lemma}

\begin{proof}
With respect to the global frame $U,V$, the endomorphism $J$ is represented
by a $2\times 2$-matrix. Writing out the condition $J^{2}=-id$ shows that
the matrix representation for $J$ has the above form. The compatibility
condition means that $d\theta (\cdot ,J\cdot )$ is a metric, so it is
represented by a positive definite matrix. Writing out this matrix shows
that $c$ is a positive function.
\end{proof}

The following is motivated by our goal to convert the torsion flow (a PDE
for tensors) into a PDE for functions. Choose real-valued functions $a$, $b$
and $c$ where $b$ and $c$ are positive. We attach super- and subscripts to
indicate the dependence on these functions. In order to keep track of
deformations of the contact form, we express all data in the given frame $%
T,U,V$. Define 
\begin{equation}  \label{eq:deformation_frame}
\begin{split}
\theta_b&=b^2 \theta,\quad \alpha_b=b \alpha-2V(b)\theta, \quad \beta_b=b
\beta+2U(b)\theta, \\
T_b&=\frac{1}{b^2}T+\frac{2V(b)}{b^3}U-\frac{2U(b)}{b^3}V, \quad U_b=\frac{1%
}{b}U, \quad V_b=\frac{1}{b}V.
\end{split}%
\end{equation}

\begin{lemma}
The vector field $T_{b}$ is the Reeb vector field for $\theta _{b}$.
Furthermore, we have $d\theta _{b}=\alpha _{b}{\wedge }\beta _{b}$.
\end{lemma}

This can be checked by plugging in the vector field into the defining
equations~\eqref{eq:defining_Reeb}. The second assertion is obtained by
writing out the terms.

\begin{remark}
If $b$ is a constant function, then the deformation from %
\eqref{eq:deformation_frame} corresponds to a $\xi$-homothetic deformation
as defined in \cite[Section 10.4]{Blair}. We take $b^2$ in $\theta_b$ to
have fewer expressions with square roots.
\end{remark}

Define a complex structure by 
\begin{equation*}
J_{abc}(U_b)=aU_b+cV_b, \quad
J_{abc}(V_b)=-\frac{1+a^2}{c}U_b-aV_b
\end{equation*}
By Lemma~\ref{lemma:cpx_structure_matrix}, this is the most general choice.

\begin{remark}
In higher dimensions, strictly pseudoconvex CR-manifolds require an integrability condition, see Definition~\ref{def:CR}, which is trivially satisfied in dimension $3$.
\end{remark}

We now compute the Tanaka-Webster connection as in Section~\ref%
{sec:pseudoherm_connection}. We use the coframe $\theta ,\theta ^{1},\theta
^{\bar{1}}$, where 
\begin{equation}
\label{eq:cpx_coframe}
\begin{split}
\theta _{abc}^{1}& =\sqrt{2c(a^{2}+1)}\left( \frac{-i}{2(a-i)}\alpha _{b}+%
\frac{i}{2c}\beta _{b}\right) , \\
\theta _{abc}^{\bar{1}}& =\sqrt{2c(a^{2}+1)}\left( \frac{i}{2(a+i)}\alpha
_{b}-\frac{i}{2c}\beta _{b}\right) .
\end{split}%
\end{equation}%
This satisfies our normalization condition $d\theta _{b}=\alpha _{b}{\wedge }%
\beta _{b}=i\theta _{b}^{1}{\wedge }\theta _{b}^{\bar{1}}$. The
corresponding eigenvectors of $J_{abc}$ are 
\begin{equation}
\label{eq:cpx_frame}
\begin{split}
Z_{1}^{abc}& =\frac{1}{\sqrt{2c(a^{2}+1)}}\left(
(a^{2}+1)U_{b}+c(a-i)V_{b}\right) , \\
Z_{\bar{1}}^{abc}& =\frac{1}{\sqrt{2c(a^{2}+1)}}\left(
a^{2}+1)U_{b}+c(a+i)V_{b}\right) .
\end{split}%
\end{equation}

\subsection{Converting the torsion flow into a system of PDE's for the
functions $a,b,c$}

To write down the equations of the torsion flow, we need the work out the
torsion tensor. We have 
\begin{equation*}
Z_{\bar 1}^{abc} \otimes \theta^1_{abc}= \left( \frac{-i(a+i)}{2}\right)
U_b\otimes \alpha_b+ \left( \frac{i(a^2+1)}{2c} \right) U_b \otimes \beta_b+
\left( \frac{-i(a+i)c}{2(a-i)} \right) V_b \otimes \alpha_b+ \left( \frac{%
i(a+i)}{2} \right) V_b \otimes \beta_b,
\end{equation*}
so we find 
\begin{equation*}
\begin{split}
A_{J,\theta}^{abc}&=A_{11}^{abc} Z_{\bar 1}^{abc} \otimes
\theta^1_{abc}+A_{\bar 1 \bar 1}^{abc} Z_{1}^{abc} \otimes \theta^{\bar
1}_{abc}=+2\re(A_{11}^{abc}Z_{\bar 1}^{abc} \otimes \theta^1_{abc} ) \\
&=+\left( \re(A_{11}^{abc})+a\im(A_{11}^{abc}) \right) U_b\otimes \alpha_b - 
\frac{\im(A_{11}^{abc})(a^2+1)}{c} U_b\otimes \beta_b \\
& \phantom{=}-\left( \re(A_{11}^{abc})\left( \frac{-2ac}{a^2+1} \right) +\im%
(A_{11}^{abc})\frac{(1-a^2)c}{a^2+1} \right) V_b\otimes \alpha_b -\left( \re%
(A_{11}^{abc})+a\im(A_{11}^{abc}) \right) V_b \otimes \beta_b.
\end{split}%
\end{equation*}
Hence the first equation of the torsion flow \eqref{eq:torsion_flow}, $\dot
J_{abc}=2A_{J,\theta}^{abc}$ is equivalent to the system 
\begin{equation}  \label{eq:system_J-part_torsion_flow}
\begin{split}
\dot a&=2\left( \re(A_{11}^{abc})+a\im(A_{11}^{abc}) \right) \\
\dot c&=-2\left( \re(A_{11}^{abc})\left( \frac{-2ac}{a^2+1} \right) +\im%
(A_{11}^{abc})\frac{(1-a^2)c}{a^2+1} \right) \\
\frac{d}{dt} {\left( -\frac{(1+a^2)}{c}\right)}&=-2 \im(A_{11}^{abc})\frac{%
a^2+1}{c} \\
\frac{d}{dt}{\ \left( -a \right) }&=-2\left( \re(A_{11}^{abc})+a\im%
(A_{11}^{abc}) \right)
\end{split}%
\end{equation}
Indeed, modulo $\theta$ we have 
\begin{equation*}
U_b\otimes \alpha_b\equiv U\otimes \alpha,\quad
U_b\otimes \beta_b\equiv U \otimes \beta,\quad
V_b\otimes \alpha_b\equiv V\otimes \alpha,\quad
V_b\otimes \beta_b\equiv V\otimes \beta,
\end{equation*}
so we obtain the above system by looking at the coefficients of $U\otimes
\alpha$, $U \otimes \beta$, $V\otimes \alpha$, and $V\otimes \beta$. This
works since these tensors are time-independent.

\begin{lemma}[A smaller system for the $J$-part of the torsion flow]
The system given by \eqref{eq:system_J-part_torsion_flow} is equivalent to
system given by

\begin{equation}  \label{eq:system_J-part_torsion_flow_simple}
\begin{split}
\dot a&=2\left( \re(A_{11}^{abc})+a\im(A_{11}^{abc}) \right) \\
\dot c&=-2\left( \re(A_{11}^{abc})\left( \frac{-2ac}{a^2+1}\right) +\im%
(A_{11}^{abc})\frac{(1-a^2)c}{a^2+1} \right) .
\end{split}%
\end{equation}
\end{lemma}

\begin{proof}
The first equation of \eqref{eq:system_J-part_torsion_flow} implies the
fourth. We now verify that the first and second equation of %
\eqref{eq:system_J-part_torsion_flow} imply the third. 
\begin{equation*}
\begin{split}
\frac{d}{dt} \left( -\frac{1+a^2}{c} \right) &=-\frac{2a\dot a}{c} +\frac{%
1+a^2}{c^2}\dot c \\
&=-\frac{4a}{c} \re(A_{11}^{abc})-\frac{4a^2}{c} \im(A_{11}^{abc}) +\frac{%
1+a^2}{c^2}\re(A_{11}^{abc})\frac{4ac}{a^2+1} -\frac{1+a^2}{c^2}\im%
(A_{11}^{abc})\frac{2c(1-a^2)}{a^2+1} \\
&=-2 \im(A_{11}^{abc}) \left( \frac{2a^2}{c}+\frac{1-a^2}{c} \right)=-2 \im%
(A_{11}^{abc}) \frac{a^2+1}{c} .
\end{split}%
\end{equation*}
\end{proof}

On the other hand, the second equation of the torsion flow %
\eqref{eq:torsion_flow} reduces to 
\begin{equation*}
\frac{d}{dt}( \theta_b )=\frac{d}{dt} (b^2) \theta=-2W^{abc} b^2 \theta,
\end{equation*}
so we can reduce the torsion flow to a system of PDE's for the functions $%
a,b,c$, giving us the following proposition.

\begin{proposition}
\label{proposition:torsion_flow_simple_PDE} Let $(M^3,\theta,J)$ be a
CR-manifold with $c_1(\xi,J)=0$. Then there exists a basis of $T^*M$, and
functions $a,b,c$ such that

\begin{itemize}
\item the complex structure $J$ can be written as $J_{abc}$ .

\item the torsion flow \eqref{eq:torsion_flow} is equivalent to the system 
\begin{equation}  \label{eq:torsion_flow_simple}
\begin{split}
\dot a&=2\left( \re(A_{11}^{abc})+a\im(A_{11}^{abc}) \right) \\
\dot c&=-2\left( \re(A_{11}^{abc})\left( \frac{-2ac}{a^2+1}\right) +\im%
(A_{11}^{abc})\frac{(1-a^2)c}{a^2+1} \right) \\
\frac{d}{dt}\left( b^2 \right)&=-2W^{abc} b^2.
\end{split}%
\end{equation}
\end{itemize}
\end{proposition}

\begin{remark}
Spatial derivatives of $a,b,c$ come in via the definition of torsion and
Webster curvature.
\end{remark}

\section{$3$-manifolds with constant structure constants and the Tanaka
connection}

\label{section:homogeneous_contact_manifolds} In this section we consider
manifolds $M^3$ that admit global $1$-forms $\omega^1,\omega^2, \omega^3$
such that

\begin{enumerate}
\item $\omega^1,\omega^2, \omega^3$ form a basis of $T^*M$.

\item The structure coefficients are constant, i.e.~$d\omega ^{i}=\sum_{j<k}c_{jk}^{i}\omega ^{j}{\wedge }\omega ^{k}$ with $c_{jk}^{i}$ constant.

\item There is a contact form $\theta $ of the form $\theta =\sum_{i}c_{i}\omega ^{i}$, where $c_{i}$ are constant.
\end{enumerate}

We shall call such a contact manifold a \textbf{homogeneous contact manifold}%
. 
This terminology is not standard, but it serves a useful purpose in this
note.
Let us point out that a related, but not equivalent notion, also referred to as homogeneous contact, was used by Perrone, \cite{pr}.

\begin{remark}
The structure coefficients are the structure constants of some $3$%
-dimensional Lie-algebra. Indeed, the dual frame $\{ X_1,X_2,X_3 \}$
satisfies 
\begin{equation*}
[X_j,X_k]=-\sum_{i}c^i_{jk}X_i,
\end{equation*}
the Lie bracket on vector fields satisfies the Jacobi identity.
We shall call this the \emph{Lie algebra of a homogeneous contact manifold}.
\end{remark}

From Lemma~\ref{lemma:global_frame_chern_class} we get immediately.

\begin{lemma}
Homogeneous contact manifolds have trivial Chern class.
\end{lemma}

Before we define a CR structure on such manifolds, we use the following
lemma to provide a better coframe. In many cases, this lemma can be improved
upon, but this version is sufficiently convenient.

\begin{lemma}
\label{lemma:nice_basis}
Let $(M,\{ \omega^i\}_i,\theta )$ be a homogeneous contact manifold. Then
there is a basis $\{ \tilde \omega^i\}_i$ such that

\begin{itemize}
\item $\tilde \omega^1$ is contact, and $c^1_{23}=1$.

\item The structure coefficients $d\tilde{\omega}^{i}=\sum_{j,k}c_{jk}^{i}%
\tilde{\omega}^{j}{\wedge }\tilde{\omega}^{k}$ satisfy 
$c_{12}^{1}=c_{13}^{1}=c^2_{12}=c^3_{13}=0$.
\end{itemize}
\end{lemma}

\begin{proof}
Choose a compatible complex structure $J$ for $(\xi,d\theta )$.
Consider the operator $h:=\frac{1}{2}\mathcal L_T J $.
From \cite[Lemma 6.2]{Blair} we see that $h$ is self-adjoint with respect to the metric $d\theta(\cdot, J \cdot )$, and we also get the identity
$$
0=\frac{1}{2}\mathcal L_T J^2=J h+hJ.
$$
Since $h$ is self-adjoint, we can find a basis of eigenvectors $X,Y$ of $h$ for $\xi$.
If the eigenvalue of $X$ is $\lambda$, then $JX$ is an eigenvector with eigenvalue $-\lambda$,
$$
hJX=-JhX=-\lambda JX,
$$
so we can assume that $Y=JX$.
We consider the Levi-Civita connection $\nabla$ for $g=\theta\otimes \theta+d\theta(\cdot,J \cdot)$.
Take the basis $e_1=T,e_2=X,e_3=JX$, where $T$ is the Reeb field of $\theta$.
Then 
\[
\begin{split}
[e_1,e_2]&=[T,X]=\nabla_T X-\nabla_XT \\
&=+JX+\lambda JX-\mu JX.
\end{split}
\]
for some $\mu\in \R$. In the last step we have used the identity (see \cite[Lemma 6.2]{Blair} )
$$
\nabla_XT=-JX-JhX.
$$
The same steps work for $[e_1,e_3]$, so we conclude that there are constants $C_1,C_2$ such that
$$
[e_1,e_2]=C_1 e_3 \quad [e_1,e_3]=C_2 e_2.
$$
Consider the dual basis $\{ \omega^i \}_i$.
Then $\omega^1$ is a contact form, and since $T$ is the Reeb field, we have
$$
0=i_{e_1}d\omega^1=c^1_{12}\omega^2+c^1_{13}\omega^3.
$$
Hence $c^1_{12}=c^1_{13}=0$, and from Lie bracket computations we see that $c^2_{12}=c^3_{13}=0$.
By rescaling, we see that the claim holds.
\end{proof}

We assume now that $\omega ^{1},\omega ^{2},\omega ^{3}$ is a basis that is
provided by this lemma. Take the basis $X_{1},X_{2},X_{3}$ that is dual to $%
\omega ^{1},\omega ^{2},\omega ^{3}$, i.e. 
\begin{equation*}
\omega ^{i}(X_{j})=\delta _{j}^{i}.
\end{equation*}%
We have $i_{X_{1}}\omega ^{1}=1$ and $i_{X_{1}}d\omega
^{1}=\sum_{k}c_{1k}^{1}\omega ^{k}=0$, so $X_{1}$ is the Reeb vector field
for $\omega ^{1}$. Note that $\xi =\ker \omega ^{1}=\Span(X_{2},X_{3})$.
Also, if $c_{23}^{1}=1$, then $d\omega ^{1}=\omega ^{2}{\wedge }\omega ^{3}$%
, so we have the right normalization convention for the setup of the Tanaka
connection described in Formula~\eqref{eq:normalized_frame} with $\theta
=\omega ^{1}$, $\alpha =\omega ^{2}$ and $\beta =\omega ^{3}$.

Choose constants $a\in \R$ and $b,c>0$, and define a complex
structure on $\xi $ (or CR-structure on $M$) following Lemma~\ref%
{lemma:cpx_structure_matrix} by 
\begin{equation}
J_{abc}=aX_{2}\otimes \omega ^{2}+cX_{3}\otimes \omega ^{2}-\frac{1+a^{2}}{c}%
X_{2}\otimes \omega ^{3}-aX_{3}\otimes \omega ^{3}.
\label{eq:J_homogeneous_contact}
\end{equation}%
We call such an endomorphism a {\bf homogeneous complex structure}, and 
we refer to a homogeneous contact manifold together with the above
complex structure as a \textbf{homogeneous CR-manifold} or a \textbf{%
homogeneous CR structure}. By direct computation, we obtain the following
result for the Tanaka-Webster connection, its torsion and the Webster curvature.

\begin{proposition}
\label{proposition:Tanaka_connection_homogeneous_contact} 
Let $(M,\{\omega
^{i}\}_{i},\theta )$ be a homogeneous contact manifold with basis provided by Lemma~\ref{lemma:nice_basis}. Fix $%
a\in \R$ and $b,c>0$ and define $J_{abc}$ as in formula~%
\eqref{eq:J_homogeneous_contact}. Then the connection form for the Tanaka
connection of the pseudohermitian manifold $(M,\theta _{b}=b^{2}\theta
,J_{abc})$ is given by 
\begin{equation*}
\omega _{1}^{\phantom{1}1}=\frac{i}{b^{2}}\left( -\frac{%
a^{2}+1}{2c}c_{12}^{3}+\frac{c}{2}c_{13}^{2}\right)
\theta _{b}+\frac{i}{b}\left( c\cdot c_{23}^{2}-a\cdot c_{23}^{3}\right)
\alpha _{b}+\frac{i}{b}\left( \frac{a^{2}+1}{c}c_{23}^{3}-a\cdot
c_{23}^{2}\right) \beta _{b}
\end{equation*}%
Its torsion is given by 
\begin{equation*}
{A_{\phantom{1}\bar{1}}^{1\phantom{1}}}^{abc}=\frac{1}{b^{2}}\left( i\frac{a^{2}+1}{2c}c_{12}^{3}-i\frac{c(a+i)}{2(a-i)}%
c_{13}^{2}\right) ,
\end{equation*}%
and its Webster curvature by 
\begin{equation*}
W^{abc}=\frac{1}{b^{2}}\left( \frac{a^{2}+1}{2c}%
c_{12}^{3}-\frac{c}{2}c_{13}^{2}-c\cdot \left(
c_{23}^{2}\right) ^{2}+2a\cdot c_{23}^{2}c_{23}^{3}-\frac{a^{2}+1}{c}\left(
c_{23}^{3}\right) ^{2}\right) .
\end{equation*}
\end{proposition}

\begin{proof}
See the computations in the appendix.
Alternatively, these computations are essentially also contained in \cite{pr}.
Note that Perrone uses a {\bf $J$-basis}, that is $e_1=T,e_2,e_3=Je_2$.
\end{proof}

We can now reduce the torsion flow for homogeneous CR-manifolds to an ODE by
plugging in the results of Proposition~\ref%
{proposition:Tanaka_connection_homogeneous_contact} into Proposition~\ref%
{proposition:torsion_flow_simple_PDE}. The general system is fairly
complicated, so we will work out some interesting case in Section~\ref%
{sec:compact_homogeneous_CR}.

We shall also consider the \textbf{normalized torsion flow} which, in
general, is given by the system 
\begin{equation}
\left\{ 
\begin{array}{l}
\partial _{t}J_{(t)}=2A_{J_{(t)},\theta _{(t)}}, \\ 
\partial _{t}\theta _{(t)}=-2(W-\bar{W})\theta _{(t)},%
\end{array}%
\right.   \label{eq:normalized_torsion_flow}
\end{equation}%
where $\bar{W}=\int_{M}W\theta {\wedge }d\theta /\int_{M}\theta {\wedge }%
d\theta $. Since the Webster curvature is constant in space for a
homogeneous CR-manifold, then second equation of the normalized is flow is
trivial. Inserting the result of Proposition~\ref%
{proposition:Tanaka_connection_homogeneous_contact} into the explicit system
provided by Proposition~\ref{proposition:torsion_flow_simple_PDE} gives the
following.

\begin{proposition}[Normalized torsion flow for homogeneous CR-manifolds]
Let $(M,\{ \omega^i \}_i,\theta)$ be a homogeneous contact manifold with $%
\theta=\omega^1$. Set $b=1$, and let $a_t,c_t$ be real valued functions Then
for a complex structure $J_{a_tbc_t}$ as defined in %
\eqref{eq:J_homogeneous_contact}, the normalized torsion flow satisfies the
ODE 
\begin{equation}  \label{eq:normalized_torsion_flow_homogeneous}
\begin{split}
\dot a_t & =c^2_{13}a_tc_t-c^3_{12}\frac{a_t^2+1}{c_t}a_t \\
\dot c_t & =c^2_{13}c_t^2+c^3_{12}(1-a_t^2).
\end{split}%
\end{equation}
\end{proposition}

\subsection{Homogeneous CR-manifolds and vanishing torsion}

Observe that homogeneous CR-manifolds need not be compact. In Section~\ref%
{sec:compact_homogeneous_CR} we give some compact examples, but we start out
by characterizing homogeneous CR-manifolds with vanishing torsion directly
in terms of the structure coefficients and coefficients for the CR-structure 
$a,b,c$.

\begin{proposition}
\label{prop:existence_torsion_free_connection} Let $(M,\{\omega
^{i}\}_{i},\theta )$ be a homogeneous contact manifold with $\theta =\omega
^{1}$ and basis provided by Lemma~\ref{lemma:nice_basis}. For $a\in \R,b=1$ and $c>0$, a complex structure $J_{abc}$ as
defined in \eqref{eq:J_homogeneous_contact} has vanishing torsion precisely
when one of the following conditions is satisfied.

\begin{itemize}
\item[1] if $a=0$, $c^2_{12}=0$, and $c=\sqrt{-\frac{c^3_{12}}{c^2_{13}}}>0$
(not imaginary).

\item[2] if $a\neq 0$, $c^2_{13}=c^2_{12}=c^3_{12}=0$.
\end{itemize}
\end{proposition}

\begin{proposition}[Convergence of the torsion flow]
\label{prop:convergence_torsion_flow}
Let $(M,\{\omega
^{i}\}_{i},\theta )$ be a homogeneous contact manifold with $\theta =\omega
^{1}$ and a basis provided by Lemma~\ref{lemma:nice_basis}.
Take $a\in \R,b=1$ and $c>0$, and let $J_{abc}$ be a complex structure as
defined in \eqref{eq:J_homogeneous_contact}.
\begin{enumerate}
\item[A] If $c^3_{12}>0$ and $c^2_{13}<0$, then the normalized torsion flow converges to the unique torsion free complex structure of Proposition~\ref{prop:existence_torsion_free_connection}
\item[R] If $c^3_{12}<0$ and $c^2_{13}>0$, then the normalized torsion flow blows up in finite time and the complex structure does not converge unless $a_0=0$ and $c_0=\sqrt{-\frac{c^3_{12}}{c^2_{13}} }$. In the latter case, the torsion vanishes and the torsion flow is constant.
\end{enumerate}

\end{proposition}

\begin{proof}
In both cases, the ODE describing the normalized torsion flow has a unique fixed point $(0,\sqrt{-\frac{c^3_{12}}{c^2_{13}} })$ in the upper half-plane with coordinates $(a,c)$.

The phase diagram for the repelling (R) case is given in Figure~\ref{fig:phasediagram}
The phase diagram for the attracting (A) case is similar, but the arrows are reversed.
\begin{figure}[htp]
\def\svgwidth{0.5\textwidth}%
\begingroup\endlinechar=-1
\resizebox{0.5\textwidth}{!}{%
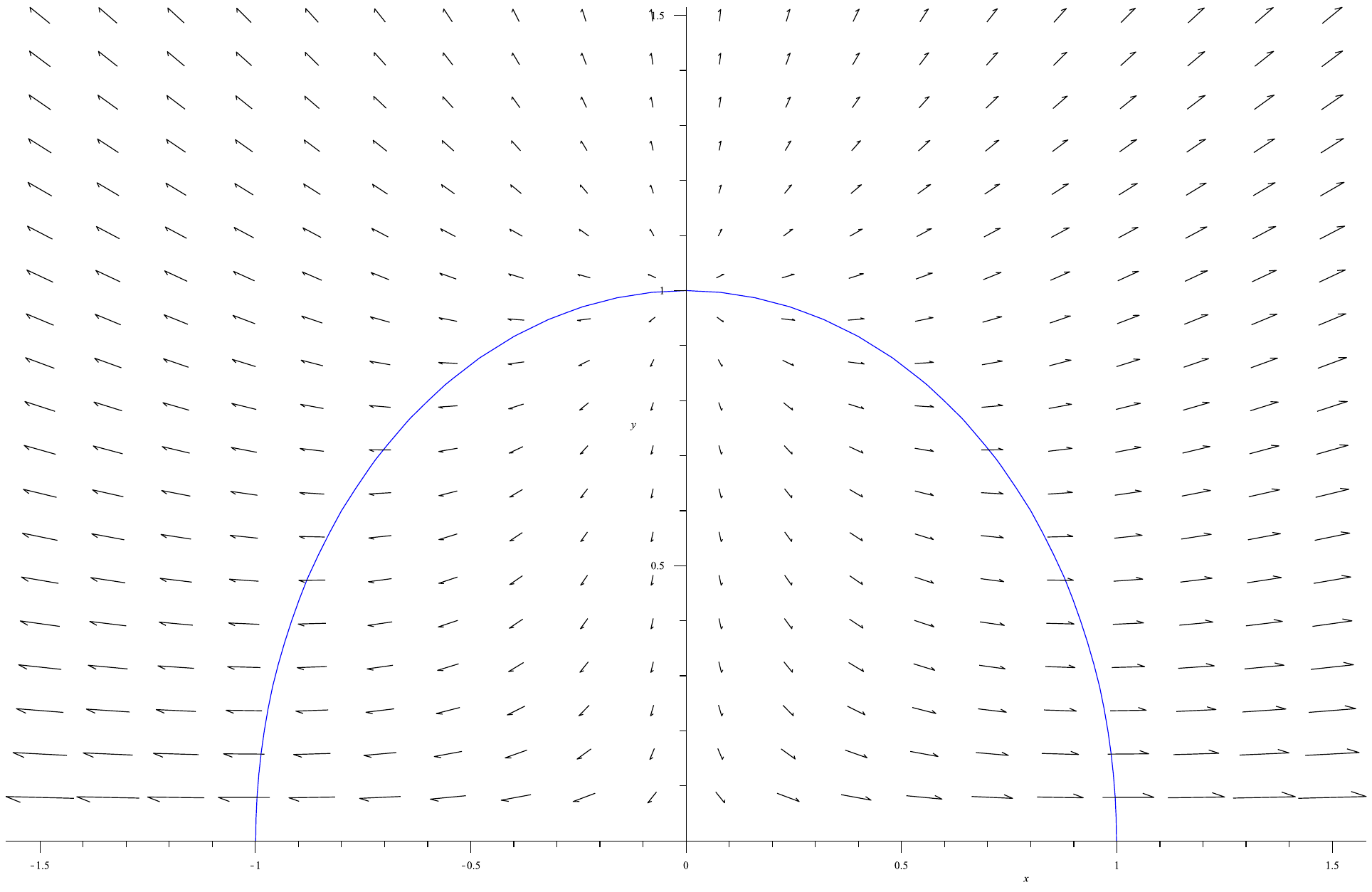%
}\endgroup
\caption{Phase diagram for the repelling case}
\label{fig:phasediagram}
\end{figure}
It is clear that the attracting case converges to $(0,\sqrt{-\frac{c^3_{12}}{c^2_{13}} })$. Indeed, the function $|a|$ is decreasing, and the only limit can be $a_\infty=0$. From the second equation we then deduce that the function $c$ converges to $c_\infty=\sqrt{-\frac{c^3_{12}}{c^2_{13}} }$.

To see that solutions in the repelling case blow up unless $(a_0,c_0)=(0,\sqrt{-\frac{c^3_{12}}{c^2_{13}} })$, we make two observations.
\begin{itemize}
\item any initial condition $(a_0,c_0)$ that starts in the set given by 
$$
|c^2_{13}|c^2_0+|c^3_{12}|a^2_0-|c^3_{12}|>0
$$
blows up in finite time. 
Indeed, the $c$-coordinate is strictly increasing in that case, so there is $t_0$ such that $c_{t_0}>1$ or the solution blows up before $t_0$.
If $c_{t_0}>1$ then the solution blows up in finite by the following argument.
We have
$$
\dot c=|c^3_{12}| c^2+|c^2_{13}|(a^2-1)\geq |c^3_{12}| c^2-|c^2_{13}|,
$$ 
which must blow up in finite time.

\item an initial condition $(a_0,c_0)$ with 
$$
|c^2_{13}|c^2_0+|c^3_{12}|a^2_0-|c^3_{12}|\leq 0 \text{ and } (a_0,c_0)\neq (0,\sqrt{-\frac{c^3_{12}}{c^2_{13}} })
$$
with $c_0>0$ exits the half-disk $\{ |c^2_{13}|c^2_0+|c^3_{12}|a^2_0-|c^3_{12}|\leq 0 \}$ in some finite time $t_1$.
Then either $c_{t_1}=0$, meaning that $a$ blows up in finite time, or we can reduce to the first case.
\end{itemize}

An alternative proof can be given by starting with a $J$-basis. Then $a_0=0$ and $c_0=1$. Such a basis is not preserved under the torsion flow, but the condition $a_t=0$ holds.
The resulting ODE is simpler to analyze.
This method is applied in Section~\ref{sec:compact_homogeneous_CR}.
\end{proof}

\begin{remark}
By Proposition~\ref{proposition:Tanaka_connection_homogeneous_contact} there is a relation between the Webster curvature and the sign of $c^3_{12}$.
However, it is \emph{not} true that repelling is equivalent to negative curvature.
\end{remark}

\subsubsection{Homogeneous CR structures on unimodular Lie groups}
We recall the definitions involved.
\begin{definition}
A Lie group $G$ is called {\bf unimodular} if the left-invariant Haar measure is also right invariant.
A Lie algebra $(\mathfrak{g},[\cdot,\cdot] )$ is called {\bf unimodular} if $Tr \ad_X=0$ for all $X\in \mathfrak g$.
\end{definition}
Note that the Lie algebra of a unimodular Lie group is unimodular.
The unimodularity of a Lie group can be used to simplify the structure coefficients of the Lie algebra.
\begin{lemma}
\label{lemma:unimodular_basis}
Let $G$ be a unimodular, $3$-dimensional Lie group admitting a homogeneous contact structure $\xi=\ker \theta$.
Then there is a basis $\{ e_i \}_i$ of the Lie algebra $\mathfrak{g}$ such that
\begin{itemize}
\item $e_1$ is the Reeb vector field, and $e_2,e_3$ lie in the contact structure $\xi$.
\item There are $\lambda,\mu\in \R$ such that $[e_1,e_2]=\lambda e_3$, $[e_1,e_3]=\mu e_2$. Furthermore, $[e_2,e_3]=e_1$.
\end{itemize}
\end{lemma}
\begin{proof}
First apply Lemma~\ref{lemma:nice_basis}.
Note that unimodularity implies that $\sum_j c^j_{ij}=0$ for all $i$. We conclude that $c^1_{i1}+c^2_{i2}+c^3_{i3}=0$.
By putting $i=2,3$ we deduce that $c^2_{32}=c^3_{23}=0$, so
$$
[e_2,e_3]=-c^1_{23}e_1=-e_1.
$$
\end{proof}
Here is a table with all possible unimodular Lie groups admitting a homogeneous contact structure,

\begin{tabular}{llll}
$c_{31}^{2}=-c_{13}^{2}$ & $c_{12}^{3}$ & Geometry & $\exists J$ with $%
A_{J,\theta }=0$ \\ \hline
+ & + & $SU(2)=S^{3}$  & yes \\ 
- & - & $\widetilde{SL(2,\R )}$  & yes \\ 
- & + & $\widetilde{SL(2,\R )}$  & no \\ 
0 & + & $E(2)$  & no \\ 
0 & - & $E(1,1)$  & no \\ 
0 & 0 & Heisenberg  & yes 
\end{tabular}

\begin{remark}
\label{rem:non-isomorphic_contact_structure} We point out that the topology
or geometry of the underlying CR-manifold does \emph{not} uniquely determine
the underlying contact structure. In particular, for some compact quotients of $\widetilde{SL(2,\R )}$ the above contact structures are not isomorphic.
Explicit examples are given in Sections~\ref{sec:pdq_str} and \ref{sec:BW_structure}.
\end{remark}

\begin{theorem}[Convergence to torsion free CR structure]
\label{thm:convergence_to_torsion_free_CR}
Let $(M,\{ \omega^i \}_i,\theta=\omega^1)$ be a homogeneous contact manifold whose Lie algebra is isomorphic to $su(2)$.
Then there is a unique homogeneous complex structure $J_{a_\infty,b=1,c_\infty}$ that is torsion free.
Moreover, for any choice of homogeneous complex structure $J_{a,b=1,c}$, the normalized torsion flow converges to this unique CR-structure $(\ker \theta, J_{a_\infty,b=1,c_\infty})$.

In particular, for any choice of homogeneous complex structure on $SU(2)$, the normalized torsion flow converges to the standard CR-structure.
\end{theorem}

\begin{example}[Rossi's examples]
We recall Rossi's examples of non-embeddable CR-manifolds.
Define the strictly plurisubharmonic function 
\[
\begin{split}
f: \C^2 & \longrightarrow \R \\
z & \longmapsto \frac{1}{2}\Vert z \Vert^2.
\end{split}
\]
Let $S^3:=f^{-1}(\frac{1}{2})$, and put $\omega^1=-df \circ i$, $\omega^2=-df \circ j$, $\omega^3=-df \circ k$, where $i,j,k$ are the standard quaternions.
With $\theta=\omega^1$, this gives $S^3$ the structure of a homogeneous contact manifold. Its structure constants are $c^1_{23}=1, c^2_{13}=-1$ and $c^3_{12}=1$.
Recall that the standard CR structure on $S^3$ is then given by $(\theta,J=i)$.
Put $\theta^1:=\frac{1}{\sqrt 2}\left( \omega^2+i \omega^3 \right)$. Then $\{ \theta,\theta^1,\theta^{\bar 1} \}$ is an admissible frame.
Following \cite{ccy} we define the CR structure via the deformed coframe
$$
\theta^1_t=\frac{1}{\sqrt{1-t^2}}\left( \theta^1-t \theta^{\bar 1}\right) .
$$
Writing this out gives
$$
\theta^1_t=\frac{1}{2}\left(
\frac{\sqrt{1-t}}{\sqrt{1+t}}\alpha
+i\frac{\sqrt{1+t}}{\sqrt{1-t}}\beta
\right)
.
$$
Comparing this with Equation~\eqref{eq:cpx_coframe} shows that the examples of Rossi are homogeneous CR structures with
$$
J_{abc}=J_{a=0,b=1/\sqrt{2},c=\frac{1-t}{1+t}}
=
\left( 
\begin{array}{cc}
0 & -\frac{1+t}{1-t} \\ 
\frac{1-t}{1+t} & 0%
\end{array}%
\right) .
$$
For $t>0$, these CR-manifolds are \emph{not} embeddable.
\end{example}
On the other hand, Theorem~\ref{thm:convergence_to_torsion_free_CR} applies to Rossi's examples, so we have.
\begin{corollary}
Under the normalized torsion flow, Rossi's examples flow to the standard CR structure on $S^3$, which is embeddable.
\end{corollary}

\subsection{Examples of compact homogeneous CR-manifolds: different CR structures on $ST^*\Sigma$}
\label{sec:compact_homogeneous_CR}
We describe the torsion flow on several geometries, namely $SU(2)$, $E(2)$, $\widetilde{SL(2,\R )}$ and Heisenberg geometry.

As an explicit, compact model covering the first three cases we consider a compact orientable surface Riemann surface $(\Sigma ,g)$. 
According to a standard theorem in
Riemannian geometry, the unit cotangent bundle $ST^{\ast }\Sigma $ admits a
canonical coframe $\omega ^{1},\omega ^{2},\omega ^{3}$ (see for instance~%
\cite{BaoChernShen}, Chapter~4.4 for the more general Finsler case with a
different ordering of the coframe) satisfying 
\begin{equation}
\begin{split}
d\omega ^{1}& =-\omega ^{2}{\wedge }\omega ^{3} \\
d\omega ^{2}& =-\omega ^{3}{\wedge }\omega ^{1} \\
d\omega ^{3}& =-K\omega ^{1}{\wedge }\omega ^{2},
\end{split}
\label{eq:coframe_unit_cotangent_bundle}
\end{equation}%
where $K$ is the Gauss curvature of $(\Sigma ,g)$. 
Assume that $g$ is a metric of constant Gauss curvature.
Then these manifolds provide models of homogeneous
contact manifolds.

\subsubsection{Homogeneous contact structure associated with the canonical contact structure ``$pdq$''}
\label{sec:pdq_str}
We consider the standard contact structure (``$pdq$'') on the unit
cotangent bundle of $(\Sigma,g)$. With respect to the canonical coframe %
\eqref{eq:coframe_unit_cotangent_bundle}, the defining form for this contact
structure is $\omega^1$.

Consider time-dependent functions $a_t,b_t,c_t$ that are constant in space,
and define the coframe 
\begin{equation*}
\theta_{b_t}={b_t}^2\omega^1,\quad \alpha_{b_t}={b_t}\omega^3,\quad
\beta_{b_t}={b_t}\omega^2.
\end{equation*}
With this ordering, we obtain the structure coefficients $c^1_{23}=1$, $%
c^2_{13}=-K$, and $c^3_{12}=1$ (take $b_0=1$), and all other coefficients
vanish. With the standard choice of complex structure $J_{abc}$, we obtain a
pseudohermitian manifold $(ST^*\Sigma,\theta_{b_t},J_{a_tb_tc_t})$. We
compute the torsion and Webster curvature with the formulas from Proposition~%
\ref{proposition:Tanaka_connection_homogeneous_contact}: 
\begin{equation}
\label{eq:torsion_webster_curvature_STSigma}
\begin{split}
{A^{1 \phantom 1}_{\phantom{1} \bar 1}}^{a_tb_tc_t} &=\frac{i}{b_t^2}\left( 
\frac{a_t^2+1}{2c_t}+\frac{c_t}{2}\frac{a_t+i}{a_t-i}K \right)
=
-\frac{a_tc_t}{a_t^2+1}K+i\left(
\frac{a_t^2+1}{2c_t}c^3_{12}+\frac{c_tK}{2}\frac{a_t^2-1}{a_t^2+1}
\right)
, \\
W^{a_tb_tc_t}&=\frac{1}{b_t^2}\left( \frac{a_t^2+1}{2c_t}+ \frac{c_t}{2}K
\right).
\end{split}%
\end{equation}
We specialize to the case that $a=0$ and substitute $B(t)=b(t)^2$.
By Proposition~\ref{proposition:torsion_flow_simple_PDE} the (unnormalized) torsion flow reduces to ODE
\begin{equation*}
\begin{split}
\dot c &= -\left( \frac{2ac}{(a^2+1)} \right)^2 \frac{K}{b^2} +\frac{1}{b^2}%
\left( \frac{a^2+1}{c} +c\frac{a^2-1}{a^2+1}K \right) \frac{(1-a^2)c}{a^2+1}
=\frac{1-c^2K}{B} \\
\dot B &= -cK-\frac{1}{c} \\
c(0) &= c_0 \\
B(0) &= (b_0)^2.
\end{split}%
\end{equation*}
If $Kc_0^2\neq 1$, then the solution to this system is given by 
\begin{equation*}
\begin{split}
c(t)&=c_0 e^{ \frac{(1-Kc_0^2)t}{b_0^2c_0}} \\
B(t)&=\frac{ Kc_0^2\left( e^{ \frac{(1-Kc_0^2)t}{b_0^2c_0}} \right)^2-1 }{%
(Kc_0^2-1)e^{ \frac{(1-Kc_0^2)t}{b_0^2c_0}}}(b_0)^2.
\end{split}%
\end{equation*}
If $Kc_0^2=1$, which can only happen if $K>0$, then 
\begin{equation*}
\begin{split}
c(t)&=c_0 \\
B(t)&=b_0^2-\frac{t}{c_0}\left( c_0^2 \cdot K+1 \right).
\end{split}%
\end{equation*}
We draw some conclusions:

\begin{itemize}
\item For $K\leq 0$, the solution exists for all time. For $K=0$ (torus
case), one has the curious property that the Webster curvature is constant.
The torsion is also constant in that case, when measured in our coframe $%
\theta_b,\alpha_b,\beta_b$. For all $K\leq 0$, the torsion flow skews the
complex structure more and more. The limit 
\begin{equation*}
\lim_{t \to\infty} c_{(t)}=\infty,
\end{equation*}
so in the limit, the complex structure blows up.

\item For $K>0$, the solution blows up in finite time because of shrinking: $%
b(t)=\sqrt{B(t)}$ reaches $0$ in finite time. The special case $Kc_0^2=1$
corresponds to vanishing torsion.
\end{itemize}

\begin{remark}
We point out that, with its canonical contact structure ``$pdq$'', only the unit
cotangent bundle of $S^2$ admits a complex structure for which the torsion
vanishes. Indeed, all other unit cotangent bundles of surfaces with constant
Gauss curvature are not K-contact, which is a necessary requirement by the
appendix of Weinstein in~\cite{ChernHamilton}.
\end{remark}

In this specialized case $a=0$, the volume-normalized flow is particularly simple. We have
\begin{equation*}
\begin{split}
\dot c_{(t)} &= 1-Kc_{(t)}^2.
\end{split}%
\end{equation*}
We see the following
\begin{enumerate}
\item if $K>0$, then there exists a torsion free complex structure, namely
for $c_\infty=1/\sqrt{K}$. We see that the torsion flow exists for all time,
and that it converges to this torsion free complex structure.

\item if $K=0$ (the torus case), then $c$ increases linearly. The flow
exists for all time, but the complex structure does not converge.

\item if $K<0$, then $c$ blows up in finite time.
Geometrically, we see by \eqref{eq:torsion_webster_curvature_STSigma} that torsion grows in norm, and the Webster curvature becomes more and more negative. Accordingly, the complex structure blows up. 
\end{enumerate}
Note that $ST^*S^2\cong SO(3)$, so alternatively we can apply Proposition~\ref{prop:convergence_torsion_flow} to the case $K>0$.

\subsubsection{Prequantization structures on $ST^*\Sigma$}
\label{sec:BW_structure}

We consider again the canonical coframe on the unit cotangent bundle with
structure coefficients as in \eqref{eq:coframe_unit_cotangent_bundle} for a
surface with constant Gauss curvature. If $\Sigma$ is \emph{not} a torus,
then we define the following coframe 
\begin{equation*}
\theta_b=b^2\omega^2,\quad \alpha_b=-bK\omega^1,\quad \beta_b=b \omega^3.
\end{equation*}
The resulting contact manifold is known as a prequantization bundle, a circle bundle over a symplectic manifold (here $\Sigma$) whose fibers are periodic Reeb orbits. 
The corresponding structure coefficients are now $%
c^1_{23}=1,c^2_{13}=-K,c^3_{12}=\frac{1}{K}$, and all other coefficients
vanish. By defining $J_{abc}$ as before, we obtain a pseudohermitian
manifold $(ST^*\Sigma,\theta_b,J_{abc})$. Its torsion and Webster curvature
are given by 
\begin{equation*}
\begin{split}
{A^{1 \phantom 1}_{\phantom{1} \bar 1}}^{abc} &=\frac{i}{b^2}\left( \frac{%
a^2+1}{2c}\frac{1}{K}+\frac{c(a+i)}{2(a-i)}K \right), \\
W^{abc}&=\frac{1}{b^2}\left( \frac{a^2+1}{2c}\frac{1}{K}+\frac{c}{2}K
\right).
\end{split}%
\end{equation*}

\begin{remark} If $\Sigma\neq S^2$, then the
resulting contact structure is \emph{not} contactomorphic to the ``$pdq$''-structure from the
previous section. Also, the contact structure is now K-contact, and we can
choose a complex structure with vanishing torsion. Indeed, choose $a=0$, and 
$c=\frac{1}{|K|}$, and the torsion tensor will vanish.
\end{remark}

As in the previous section we specialize to the case that $a=0$ and substitute $B(t)=b(t)^2$.
By Proposition~\ref{proposition:torsion_flow_simple_PDE} the (unnormalized) torsion flow reduces to ODE
\begin{equation*}
\begin{split}
\dot c &= \frac{-Kc^2+\frac{1}{K}}{B} \\
\dot B &= -\left(  \frac{1}{Kc} +cK \right) \\
c(0) &= c_0 \\
B(0) &= (b_0)^2.
\end{split}%
\end{equation*}
The solution is given by
\[
\begin{split}
c(t)&=c_0e^{\frac{1-K^2c_0^2}{B_0c_0K}t} \\
B(t)&=B_0\frac{1-K^2c_0^2e^{\frac{1-K^2c_0^2}{B_0c_0K}2t}}{(1-K^2c_0^2)e^{\frac{1-K^2c_0^2}{B_0c_0K}t}}
\end{split}
\]
If we start the flow at $a_0=0$ and $c_0=\frac{1}{|K|}$, then we have vanishing torsion, and the torsion flow just contracts or expands depending on the sign of the Webster curvature.
We have
$$
B(t)=B_0-\frac{2|K|}{K}t.
$$
The normalized torsion for these homogeneous contact manifolds are covered by Proposition~\ref{prop:convergence_torsion_flow}.

\subsubsection{Heisenberg geometry}
As an explicit, compact example with Heisenberg geometry, consider the $2$-torus with standard symplectic form $(T^2,\Omega=d\phi_1\w d\phi_2)$.
There is a principal circle bundle $p:P\to T^2$ whose connection form $\theta$ satisfies $d\theta=p^*\Omega$.
We see that $(P,\theta,\alpha=d\phi_1,\beta=d\phi_2)$ is a homogeneous contact manifold of Heisenberg type. Indeed, all structure coefficients except $c^1_{23}$ vanish.

Hence any homogeneous CR-structure has vanishing torsion and Webster curvature. It follows that the torsion flow is constant, so this is an explicit example of a torsion soliton, namely a steady breather, see Corollary~\ref{c51}, case (i).

\section{Entropy functionals}
\label{sec:entropy}
The following section discusses entropy functionals on a closed $3$-dimensional pseudohermitian manifold $(M,J,\theta )$.

\subsection{The Entropy $\mathcal{F}$-Functional}

\label{sec:F_functional}

Let $(M,J,\theta )$ be a closed pseudohermitian $3$-manifold. In this
section, we study the monotonicity property of the $\mathcal{F}$-functional 
\begin{equation*}
\begin{array}{c}
\mathcal{F}(J_{(t)},\theta _{(t)},\varphi _{(t)})=\int_{M}(W+\left\vert
\nabla _{b}\varphi \right\vert _{J,\theta }^{2})e^{-\varphi }d\mu%
\end{array}%
\end{equation*}%
with the constraint%
\begin{equation*}
\begin{array}{c}
\int_{M}e^{-\varphi }d\mu =1%
\end{array}%
\end{equation*}%
under the coupled torsion flow \eqref{eq:coupled_torsion_flow}.

\proof
We compute $\frac{\partial}{\partial t} \nabla_b \phi$ with Equation~\eqref{eq:variationZ1} and use the result to obtain
\begin{equation}
\begin{array}{c}
\text{{\small $\frac{\partial }{\partial t}$}}\left \vert \nabla _{b}\varphi
\right \vert _{J,\theta }^{2}=4\re \left( iE_{\overline{1}\overline{1}%
}\varphi _{1}\varphi _{1}\right) +2\langle \nabla _{b}\varphi ,\nabla
_{b}\varphi _{t}\rangle _{J,{\theta }}-2\eta _{(t)}\left \vert \nabla
_{b}\varphi \right \vert _{J,\theta }^{2}.
\end{array}
\label{18c}
\end{equation}%
By \eqref{eq:coupled_torsion_flow} we find
\begin{equation}
\begin{array}{c}
\frac{\partial }{\partial t}d\mu =4\eta _{(t)}d\mu .%
\end{array}
\label{18a}
\end{equation}%
Use these formulas together with a variation formula for the Webster curvature to compute the variation of the $\mathcal F$-functional,
\begin{equation}
\begin{array}{l}
-\frac{1}{2}\text{{\small $\frac{d}{dt}\mathcal{F}(J_{(t)},\theta
_{(t)},\varphi _{(t)})$}} \\ 
=-\int_{M}\eta _{(t)}[{W+}\left \vert \nabla _{b}\varphi \right \vert
_{J,\theta }^{2}]e^{-\varphi }d\mu +2\int_{M}(\Delta _{b}\eta
_{(t)})e^{-\varphi }d\mu \\ 
\ \ +\frac{1}{2}\int_{M}(2{\Delta _{b}}\varphi -\left \vert \nabla
_{b}\varphi \right \vert _{J,\theta }^{2}+W)\varphi _{t}e^{-\varphi }d\mu \\ 
\ \ -2\int_{M}\re \left( iE_{\overline{1}\overline{1}}\varphi
_{1}\varphi _{1}\right) e^{-\varphi }d\mu -\int_{M}\re \left( iE_{11,%
\overline{1}\overline{1}}-A_{11}E_{\overline{1}\overline{1}}\right)
e^{-\varphi }d\mu \\ 
=\int_{M}(\frac{1}{2}\varphi _{t}-\eta _{(t)})(2{\Delta _{b}}\varphi -\left
\vert \nabla _{b}\varphi \right \vert _{J,\theta }^{2}+W)e^{-\varphi }d\mu
\\ 
\ \ +\int_{M}\re [\left( A_{11}-i\varphi _{11}-i\varphi _{1}\varphi
_{1}\right) E_{\overline{1}\overline{1}}]e^{-\varphi }d\mu \\ 
=\int_{M}\eta _{(t)}(2{\Delta _{b}}\varphi -\left \vert \nabla _{b}\varphi
\right \vert _{J,\theta }^{2}+W)e^{-\varphi }d\mu \\ 
\ \ \ +\int_{M}\re [\left( A_{11}-i\varphi _{11}-i\varphi _{1}\varphi
_{1}\right) E_{\overline{1}\overline{1}}]e^{-\varphi }d\mu .\ 
\end{array}
\label{18}
\end{equation}%
We first set $E_{11}=e^{\varphi }(A_{11}-i\varphi _{11}-i\varphi _{1}\varphi
_{1})$ and $\eta _{(t)}=e^{\varphi }(2{\Delta _{b}}\varphi -\left \vert
\nabla _{b}\varphi \right \vert _{J,\theta }^{2}+W),$ then 
\begin{equation*}
\begin{array}{ccl}
-\frac{1}{2}\text{{\small $\frac{d}{dt}\mathcal{F}(J_{(t)},\theta
_{(t)},\varphi _{(t)})$}} & = & \int_{M}(2{\Delta _{b}}\varphi -\left \vert
\nabla _{b}\varphi \right \vert _{J,\theta }^{2}+W)^{2}d\mu \\ 
&  & +\int_{M}|A_{11}-i\varphi _{11}-i\varphi _{1}\varphi _{1}|^{2}d\mu \\ 
& \geq & 0.%
\end{array}%
\end{equation*}

The monotonicity formula is strict unless%
\begin{equation*}
A_{11}-i\varphi _{11}-i\varphi _{1}\varphi _{1}=0\text{ \ and \ }2{\Delta
_{b}}\varphi -\left \vert \nabla _{b}\varphi \right \vert _{J,\theta
}^{2}+W=0.
\end{equation*}%
Moreover, up to a contact transformation $\widetilde{\theta }=e^{-\varphi
}\theta $%
\begin{equation*}
\widetilde{A_{11}}=0,\ \ \widetilde{W}=0.
\end{equation*}%
This completes the proof of Theorem \ref{t2}. 
\endproof%

\subsection{The Entropy $\mathcal{W}^{\pm }$-Functionals}

\label{sec:Wfunctional}

We study the monotonicity property of the $\mathcal{W}^{+}$-functional 
\begin{equation*}
\begin{array}{c}
\mathcal{W}^{+}(J_{(t)},\theta _{(t)},\varphi _{(t)},\tau
_{(t)})=\int_{M}[\tau (W+\left \vert \nabla _{b}\varphi \right \vert
_{J,\theta }^{2})+\frac{1}{2}\varphi -1](4\pi \tau )^{-2}e^{-\varphi }d\mu%
\end{array}%
\end{equation*}%
with the constraint 
\begin{equation*}
\begin{array}{c}
\int_{M}(4\pi \tau )^{-2}e^{-\varphi }d\mu =1%
\end{array}%
\end{equation*}%
under the coupled torsion flow (\ref{2008-3}).

\proof
Following the same computations as in the proof of Theorem \ref{t2}, we can
derive that%
\begin{equation*}
\begin{array}{l}
\text{{\small $\frac{d}{dt}$}}\int_{M}[\tau (W+\left \vert \nabla
_{b}\varphi \right \vert _{J,\theta }^{2})-1](4\pi \tau )^{-2}e^{-\varphi
}d\mu \\ 
=\int_{M}\frac{d\tau }{dt}(W+\left \vert \nabla _{b}\varphi \right \vert
_{J,\theta }^{2})(4\pi \tau )^{-2}e^{-\varphi }d\mu \\ 
\ \ \text{\ }-2\tau \int_{M}[\eta _{(t)}({W+}\left \vert \nabla _{b}\varphi
\right \vert _{J,\theta }^{2})+2{\Delta _{b}}\eta _{(t)}-4\langle \nabla
_{b}\varphi ,\nabla _{b}\eta _{(t)}\rangle ](4\pi \tau )^{-2}e^{-\varphi
}d\mu \\ 
\ \ \text{\ }+2\tau \int_{M}\re (iE_{11,\overline{1}\overline{1}%
}-A_{11}E_{\overline{1}\overline{1}}+2iE_{\overline{1}\overline{1}}\varphi
_{1}\varphi _{1})(4\pi \tau )^{-2}e^{-\varphi }d\mu \\ 
\text{ \ \ }+\int_{M}[\tau (W+\left \vert \nabla _{b}\varphi \right \vert
_{J,\theta }^{2}-1](4\eta _{(t)}-2\tau ^{-1}\frac{d\tau }{dt}-\varphi
_{t})(4\pi \tau )^{-2}e^{-\varphi }d\mu \\ 
=2\int_{M}(W+\left \vert \nabla _{b}\varphi \right \vert _{J,\theta
}^{2})(4\pi \tau )^{-2}e^{-\varphi }d\mu \\ 
\ \ \text{\ }-2\tau \int_{M}\eta _{(t)}(2{\Delta _{b}}\varphi -\left \vert
\nabla _{b}\varphi \right \vert _{J,\theta }^{2}+W)(4\pi \tau
)^{-2}e^{-\varphi }d\mu \\ 
\ \ \ -2\tau \int_{M}\re[\left( A_{11}-i\varphi _{11}-i\varphi
_{1}\varphi _{1}\right) E_{\overline{1}\overline{1}}](4\pi \tau
)^{-2}e^{-\varphi }d\mu \\ 
=2\int_{M}({W+}\left \vert \nabla _{b}\varphi \right \vert _{J,\theta
}^{2})(4\pi \tau )^{-2}e^{-\varphi }d\mu -2\int_{M}(\eta _{(t)}-\tau
^{-1})(4\pi \tau )^{-2}e^{-\varphi }d\mu \\ 
\ \ \ -2\int_{M}(2{\Delta _{b}}\varphi -\left \vert \nabla _{b}\varphi
\right \vert _{J,\theta }^{2}+W)(4\pi \tau )^{-2}e^{-\varphi }d\mu \\ 
\ \ \ -2\tau \int_{M}(\eta _{(t)}-\tau ^{-1})(2{\Delta _{b}}\varphi -\left
\vert \nabla _{b}\varphi \right \vert _{J,\theta }^{2}+W-\tau ^{-1})(4\pi
\tau )^{-2}e^{-\varphi }d\mu \\ 
\ \ \ -2\tau \int_{M}\re[\left( A_{11}-i\varphi _{11}-i\varphi
_{1}\varphi _{1}\right) E_{\overline{1}\overline{1}}](4\pi \tau
)^{-2}e^{-\varphi }d\mu \\ 
=-2\int_{M}(\eta _{(t)}-\tau ^{-1})(4\pi \tau )^{-2}e^{-\varphi }d\mu \\ 
\ \ \ -2\tau \int_{M}(\eta _{(t)}-\tau ^{-1})(2{\Delta _{b}}\varphi -\left
\vert \nabla _{b}\varphi \right \vert _{J,\theta }^{2}+W-\tau ^{-1})(4\pi
\tau )^{-2}e^{-\varphi }d\mu \\ 
\ \ \ -2\tau \int_{M}\re[\left( A_{11}-i\varphi _{11}-i\varphi
_{1}\varphi _{1}\right) E_{\overline{1}\overline{1}}](4\pi \tau
)^{-2}e^{-\varphi }d\mu \\ 
=-2\int_{M}(2{\Delta _{b}}\varphi -\left \vert \nabla _{b}\varphi \right
\vert _{J,\theta }^{2}+W-\tau ^{-1})(4\pi \tau )^{-2}e^{-\varphi }d\mu \\ 
\ \ \ -2\tau \int_{M}(2{\Delta _{b}}\varphi -\left \vert \nabla _{b}\varphi
\right \vert _{J,\theta }^{2}+W-\tau ^{-1})^{2}(4\pi \tau )^{-2}e^{-\varphi
}d\mu \\ 
\ \ \text{\ }\ -2\tau \int_{M}|A_{11}-i\varphi _{11}-i\varphi _{1}\varphi
_{1}|^{2}(4\pi \tau )^{-2}e^{-\varphi }d\mu .%
\end{array}%
\end{equation*}%
Here we have used the following identities : 
\begin{equation*}
\begin{array}{c}
\int_{M}({\Delta _{b}}\eta _{(t)})(4\pi \tau )^{-2}e^{-\varphi }d\mu
=\int_{M}\eta _{(t)}(\left \vert \nabla _{b}\varphi \right \vert _{J,\theta
}^{2}-{\Delta _{b}}\varphi )(4\pi \tau )^{-2}e^{-\varphi }d\mu%
\end{array}%
\end{equation*}%
and 
\begin{equation*}
\int_{M}({W+}\left \vert \nabla _{b}\varphi \right \vert _{J,\theta
}^{2})(4\pi \tau )^{-2}e^{-\varphi }d\mu =\int_{M}(2{\Delta _{b}}\varphi
-\left \vert \nabla _{b}\varphi \right \vert _{J,\theta }^{2}+W)(4\pi \tau
)^{-2}e^{-\varphi }d\mu .
\end{equation*}

On the other hand,%
\begin{equation*}
\begin{array}{l}
\frac{1}{2}\frac{d}{dt}\int_{M}\varphi (4\pi \tau )^{-2}e^{-\varphi }d\mu \\ 
=2\int_{M}(2{\Delta _{b}}\varphi -\left \vert \nabla _{b}\varphi \right
\vert _{J,\theta }^{2}+W-\tau ^{-1})(4\pi \tau )^{-2}e^{-\varphi }d\mu .%
\end{array}%
\end{equation*}%
It follows that 
\begin{equation*}
\begin{array}{l}
-\frac{1}{2}\text{{\small $\frac{d}{dt}\mathcal{W}^{+}(J_{(t)},\theta
_{(t)},\varphi _{(t)},\tau _{(t)})$}} \\ 
=\tau \int_{M}(2{\Delta _{b}}\varphi -\left \vert \nabla _{b}\varphi \right
\vert _{J,\theta }^{2}+W-\tau ^{-1})^{2}(4\pi \tau )^{-2}e^{-\varphi }d\mu
\\ 
\text{ \ \ }+\tau \int_{M}|A_{11}-i\varphi _{11}-i\varphi _{1}\varphi
_{1}|^{2}(4\pi \tau )^{-2}e^{-\varphi }d\mu .%
\end{array}%
\end{equation*}%
Moreover, the monotonicity formula is strict unless%
\begin{equation*}
A_{11}-i\varphi _{11}-i\varphi _{1}\varphi _{1}=0\text{\ \ and\ \ }2{\Delta
_{b}}\varphi -\left \vert \nabla _{b}\varphi \right \vert _{J,\theta
}^{2}+W-\tau ^{-1}=0.
\end{equation*}%
This completes the proof of Theorem \ref{t3}. 
\endproof%

Next we study the monotonicity property of $\mathcal{W}^{-}$-functional 
\begin{equation*}
\begin{array}{c}
\mathcal{W}^{-}(J_{(t)},\theta _{(t)},\varphi _{(t)},\tau
_{(t)})=\int_{M}[\tau (W+\left \vert \nabla _{b}\varphi \right \vert
_{J,\theta }^{2})-\frac{1}{2}\varphi +1](4\pi \tau )^{-2}e^{-\varphi }d\mu%
\end{array}%
\end{equation*}%
with the constraint%
\begin{equation*}
\begin{array}{c}
\int_{M}(4\pi \tau )^{-2}e^{-\varphi }d\mu =1%
\end{array}%
\end{equation*}%
under the coupled torsion flow (\ref{2008-2}).

\proof
Following the same computations as in the proof of Theorem \ref{t3}, we can
derive that%
\begin{equation*}
\begin{array}{l}
\text{{\small $\frac{d}{dt}$}}\int_{M}[\tau (W+\left \vert \nabla
_{b}\varphi \right \vert _{J,\theta }^{2})+1](4\pi \tau )^{-2}e^{-\varphi
}d\mu \\ 
=\int_{M}\frac{d\tau }{dt}({W+}\left \vert \nabla _{b}\varphi \right \vert
_{J,\theta }^{2})(4\pi \tau )^{-2}e^{-\varphi }d\mu \\ 
\ \ \text{\ }-2\tau \int_{M}[\eta _{(t)}({W+}\left \vert \nabla _{b}\varphi
\right \vert _{J,\theta }^{2})+2{\Delta _{b}}\eta _{(t)}-4\langle \nabla
_{b}\varphi ,\nabla _{b}\eta _{(t)}\rangle ](4\pi \tau )^{-2}e^{-\varphi
}d\mu \\ 
\ \ \text{\ }+2\tau \int_{M}\re(iE_{11,\overline{1}\overline{1}%
}-A_{11}E_{\overline{1}\overline{1}}+2iE_{\overline{1}\overline{1}}\varphi
_{1}\varphi _{1})(4\pi \tau )^{-2}e^{-\varphi }d\mu \\ 
\ \ \text{\ }+\int_{M}[\tau ({W+}\left \vert \nabla _{b}\varphi \right \vert
_{J,\theta }^{2})-1](4\eta _{(t)}-2\tau ^{-1}\frac{d\tau }{dt}-\varphi
_{t})(4\pi \tau )^{-2}e^{-\varphi }d\mu \\ 
=-2\int_{M}({W+}\left \vert \nabla _{b}\varphi \right \vert _{J,\theta
}^{2})(4\pi \tau )^{-2}e^{-\varphi }d\mu \\ 
\ \text{\ }\ -2\tau \int_{M}\eta _{(t)}(2{\Delta _{b}}\varphi -\left \vert
\nabla _{b}\varphi \right \vert _{J,\theta }^{2}+W)(4\pi \tau
)^{-2}e^{-\varphi }d\mu \\ 
\ \ \text{\ }-2\tau \int_{M}\re[\left( A_{11}-i\varphi _{11}-i\varphi
_{1}\varphi _{1}\right) E_{\overline{1}\overline{1}}](4\pi \tau
)^{-2}e^{-\varphi }d\mu \\ 
=2\int_{M}(2{\Delta _{b}}\varphi -\left \vert \nabla _{b}\varphi \right
\vert _{J,\theta }^{2}+W+\tau ^{-1})(4\pi \tau )^{-2}d\mu \\ 
\ \ \text{\ }-2\tau \int_{M}(2{\Delta _{b}}\varphi -\left \vert \nabla
_{b}\varphi \right \vert _{J,\theta }^{2}+W+\tau ^{-1})^{2}(4\pi \tau
)^{-2}d\mu \\ 
\text{ \ \ }-2\tau \int_{M}|A_{11}-i\varphi _{11}-i\varphi _{1}\varphi
_{1}|^{2}(4\pi \tau )^{-2}d\mu .%
\end{array}%
\end{equation*}

On the other hand,%
\begin{equation*}
\begin{array}{l}
-\frac{1}{2}\frac{d}{dt}\int_{M}\varphi (4\pi \tau )^{-2}e^{-\varphi }d\mu
\\ 
=-2\int_{M}(2{\Delta _{b}}\varphi -\left \vert \nabla _{b}\varphi \right
\vert _{J,\theta }^{2}+W+\tau ^{-1})(4\pi \tau )^{-2}d\mu .%
\end{array}%
\end{equation*}%
It follows that 
\begin{equation*}
\begin{array}{l}
-\frac{1}{2}\text{{\small $\frac{d}{dt}\mathcal{W}^{-}(J_{(t)},\theta
_{(t)},\varphi _{(t)},\tau _{(t)})$}} \\ 
=\tau \int_{M}(2{\Delta _{b}}\varphi -\left \vert \nabla _{b}\varphi \right
\vert _{J,\theta }^{2}+W+\tau ^{-1})^{2}(4\pi \tau )^{-2}d\mu \\ 
\text{ \ \ }+\tau \int_{M}|A_{11}-i\varphi _{11}-i\varphi _{1}\varphi
_{1}|^{2}(4\pi \tau )^{-2}d\mu .%
\end{array}%
\end{equation*}%
Moreover, the monotonicity formula is strict unless%
\begin{equation*}
A_{11}-i\varphi _{11}-i\varphi _{1}\varphi _{1}=0\text{\ \ and\ \ }2{\Delta
_{b}}\varphi -\left \vert \nabla _{b}\varphi \right \vert _{J,\theta
}^{2}+W+\tau ^{-1}=0.
\end{equation*}%
This completes the proof of Theorem \ref{t4}. 
\endproof%

\section{Appendix}

\subsection{Computations}

\begin{proof}[Proof of Proposition~\protect\ref%
{proposition:Tanaka_connection_homogeneous_contact}]
As in Formula~\eqref{eq:deformation_frame}, put 
\begin{equation*}
\theta_b=b^2\omega^1,\quad \alpha_b=b \omega^2,\quad \beta_b=b\omega^3,\quad
T_b=\frac{1}{b^2}X_1,\quad U_b=\frac{1}{b}X_2,\quad V_b=\frac{1}{b}X_3.
\end{equation*}
We first work out some general formulas. Given the second equation of %
\eqref{eq:defining_Tanaka_Webster}, we can assume without loss of generality
that 
\begin{equation}  \label{eq:TW_connection_form}
\omega^{\phantom{1}1}_{1\phantom{1}}=ic_\theta \theta+ic_Z \theta^1+i\bar
c_Z \theta^1,
\end{equation}
with $c_\theta$ a real function. Now we insert our frame to get explicit
equations 
\begin{equation}  \label{eq:structure_functions_general_Tanaka_connection}
\begin{split}
d\theta^1(Z_1,Z_{\bar 1})&= \omega^{\phantom{1}1}_{1\phantom{1}}(Z_{\bar
1})=i\bar c_{Z} \\
d\theta^1(T,Z_1)&=-ic_\theta \\
d\theta^1(T,Z_{\bar 1})&=A^{1\phantom{1}}_{\phantom{1}\bar 1}.
\end{split}%
\end{equation}

The Webster curvature is determined by the second structure equation, 
\begin{equation*}
d\omega _{1\phantom{1}}^{\phantom{1}1}=W\theta ^{1}{\wedge }\theta ^{\bar{1}%
}+2i\im(A_{\phantom{1}1,\bar{1}}^{\bar{1}}\theta ^{1}{\wedge }\theta ).
\end{equation*}%
Wedging this equation with $\theta $ gives $\theta {\wedge }d\omega _{1%
\phantom{1}}^{\phantom{1}1}=W\theta {\wedge }\theta ^{1}{\wedge }\theta ^{%
\bar{1}}$, which can be rewritten by using \eqref{eq:TW_connection_form} to
write out $d\omega _{1\phantom{1}}^{\phantom{1}1}$. We find $\theta {\wedge }%
d\omega _{1\phantom{1}}^{\phantom{1}1}=\left( -c_{\theta }-2|c_{Z}|^{2}-2\im%
(Z_{1}(\bar{c}_{Z})\,)\right) \theta {\wedge }\theta ^{1}{\wedge }\theta ^{%
\bar{1}}$, and conclude 
\begin{equation}
W=-c_{\theta }-2|c_{Z}|^{2}-2\im(Z_{1}(\bar{c}_{Z})\,).
\label{eq:Formula_Webster_curvature_frame}
\end{equation}

We now determine these coefficients.
For $c_{\theta }^{abc}$, we use the second equation from~%
\eqref{eq:structure_functions_general_Tanaka_connection} to get 
\[
\begin{split}
-ic_{\theta _{b}}^{abc} &=-\frac{i(a+i)}{2}d\alpha _{b}(T_{b},U_{b})+\frac{%
a^{2}+1}{2c}id\beta _{b}(T_{b},U_{b})-\frac{ci}{2}d\alpha _{b}(T_{b},V_{b})+%
\frac{i}{2}(a-i)d\beta _{b}(T_{b},V_{b}) \\
&=
-\frac{i(a+i)}{2b^2} c^2_{12}
+i\frac{a^2+1}{2c b^2}c^3_{12}
-\frac{ci}{2b^2}c^2_{13}
+\frac{i}{2b^2}(a-i)c^3_{13}.
\end{split}
\]
Use Lemma~\ref{lemma:nice_basis} to see that $c^2_{12}=c^3_{13}=0$.
For the $c_Z$-component, we use the first equation from %
\eqref{eq:structure_functions_general_Tanaka_connection} we compute 
\begin{equation*}
\begin{split}
i\bar{c_Z}&=d \theta^1(Z_1,Z_{\bar 1} ) \\
&=\frac{1}{\sqrt{2c(a^2+1)}}\left( -\frac{i}{2(a-i)}d \alpha_b(U_b,V_b)
\left( (a^2+1)c(a+i)-c(a-i)(a^2+1) \right) \right. \\
&\phantom{=}\left. +\frac{i}{2c} d\beta_b(U_b,V_b) \left(
(a^2+1)c(a+i)-c(a-i)(a^2+1) \right) \right) \\
&=\frac{1}{\sqrt{2c(a^2+1)}}\left( c(a+i)c^2_{23}-(a^2+1)c^3_{23} \right) .
\end{split}%
\end{equation*}

By the above formulas, these coefficients determine the connection form $\omega_1^{\phantom{1}1}$ and the curvature $W$.
For the torsion we use formulas \eqref{eq:cpx_frame}, \eqref{eq:cpx_coframe} and \eqref{eq:structure_functions_general_Tanaka_connection}, and find
\[
\begin{split}
A^1_{\phantom{1}\bar 1}& =d\theta^1_b(T_b,Z_{b,\bar 1} ) \\
&=\frac{i}{2(a-i)}d\alpha_b(T_b(a^2+1)U_b+c(a+i)V_b)
+\frac{i}{2c}d\beta_b(T_b,(a^2+1)U_b+c(a+i)V_b) \\
&=\frac{i}{2(a-i)b^2}\left(
(a^2+1)c^2_{12}+c(a+i)c^2_{13}
\right)
+
\frac{i}{2cb^2}\left(
(a^2+1)c^3_{12}+c(a+i)c^3_{13}.
\right)
\end{split}
\]
Combining this with $c^2_{12}=c^3_{13}=0$ gives the desired expression for the torsion.
\end{proof}

\subsection{Variation formulas}
Let $\theta _{(t)}$ be a family of smooth contact forms and $J_{(t)}$ be a
family of CR structures on $(M,J,\theta )$. We consider the following flow
on a closed pseudohermitian $(2n+1)$-manifold $(M,J,\theta )\times \lbrack
0,T)$: 
\begin{equation}
\left\{ 
\begin{array}{l}
\partial _{t}J_{(t)}=2E, \\ 
\partial _{t}\theta _{(t)}=2\eta _{(t)}\theta _{(t)}.%
\end{array}%
\right.  \label{PTF}
\end{equation}%
Here $J=i\theta ^{\alpha }\otimes Z_{\alpha }-i\theta ^{\overline{\alpha }%
}\otimes Z_{\overline{\alpha }}$ and $E=E_{\alpha }{}^{\overline{\beta }%
}\theta ^{\alpha }\otimes Z_{\overline{\beta }}+E_{\overline{\alpha }%
}{}^{\beta }\theta ^{\overline{\alpha }}\otimes Z_{\beta }$.

We start by deriving some evolution equations under the general flow \eqref%
{PTF} before specifying to the torsion flow, for which $E=A_J$ (the torsion tensor), and $\eta=-W$ (the Webster curvature).
All computations will be done in a local frame.
Fix a unit-length local frame $%
\{Z_{\alpha }\}$ and let $\{\theta ^{\alpha }\}$ be its dual admissible $1$%
-form. 
Let $Z_{\alpha (t)},\theta _{(t)}^{\alpha }$ denote a unit-length
frame and dual admissible $1$-form with respect to $(J_{(t)},\theta _{(t)})$.
Since $\theta ^{\alpha }(Z_{\beta (t)})$ is a positive real function, we can write $\overset{\centerdot }{Z_{\alpha }}=F_{\alpha
}{}^{\beta }Z_{\beta }+G_{\alpha }{}^{\overline{\beta }}Z_{\overline{\beta }%
} $ where $F_{\alpha }{}^{\beta }$ are real and $G_{\alpha }{}^{\overline{%
\beta }}$ are complex. The fact that $Z_{\alpha (t)}$ is an orthonormal
frame means that 
\begin{equation*}
\delta _{\alpha \beta }=-id\theta _{(t)}(Z_{\alpha (t)}\wedge Z_{\overline{%
\beta }(t)}).
\end{equation*}%
By differentiating and substituting the above expression for $\overset{\centerdot }{Z_{\alpha }},$
we obtain $F_{\alpha }{}^{\beta }=-\eta \delta _{\alpha }^{\beta }.$
By differentiating $J_{(t)}Z_{\alpha (t)}=iZ_{\alpha (t)}$ we find 
\begin{equation*}
0=\overset{\centerdot }{J}Z_{\alpha }+J\overset{\centerdot }{Z_{\alpha }}-i%
\overset{\centerdot }{Z_{\alpha }}=2E_{\alpha }{}^{\overline{\beta }}Z_{%
\overline{\beta }}-2iG_{\alpha }{}^{\overline{\beta }}Z_{\overline{\beta }},
\end{equation*}%
so 
\begin{equation}
\overset{\centerdot }{Z_{\alpha }}=-\eta Z_{\alpha }-iE_{\alpha }{}^{%
\overline{\beta }}Z_{\overline{\beta }}.  
\label{eq:variationZ1}
\end{equation}%
Now differentiate the identities%
\begin{equation*}
d\theta _{(t)}=ih_{\alpha \bar{\beta}}\theta _{(t)}^{\alpha }\wedge \theta
_{(t)}^{\overline{\beta }},\text{ \ }\theta _{(t)}^{\alpha }(Z_{\beta
(t)})=\delta _{\beta }^{\alpha },\text{ \ \textrm{and} \ }\theta
_{(t)}^{\alpha }(Z_{\overline{\beta }(t)})=0,
\end{equation*}%
to deduce that 
\begin{equation}
\overset{\centerdot }{{\theta }}{{^{\alpha }}}=2i\eta ^{\alpha }\theta +\eta
\theta ^{\alpha }-iE^{\alpha }{_{\overline{\beta }}}\theta ^{\overline{\beta 
}}.  
\label{24}
\end{equation}
Now we differentiate \eqref{eq:defining_Tanaka_Webster} to obtain%
\begin{equation}
d\overset{\centerdot }{{\theta }}{{^{\alpha }}}=\overset{\centerdot }{{%
\theta }}{{^{\gamma }}\wedge }{\omega _{\gamma }}^{\alpha }+\theta ^{\gamma }%
{\wedge }\overset{\centerdot }{{\omega }}{_{\gamma }}^{\alpha }+\overset{%
\centerdot }{{A}}_{\overline{\alpha }\overline{\gamma }}{\theta }{\wedge }{%
\theta ^{\overline{\gamma }}+{A}}_{\overline{\alpha }\overline{\gamma }}%
\overset{\centerdot }{{\theta }}{{\wedge }{\theta ^{\overline{\gamma }}+{A}_{%
\overline{\alpha }\overline{\gamma }}\theta }{\wedge }}\overset{\centerdot }{%
{\theta }}{{^{\overline{\gamma }}}.}  \label{25}
\end{equation}%
Since we will derive an identity involving tensors, we will take an adapted frame satisfying ${\omega _{\gamma }}^{\alpha
}=0$ at a point.
Plug in \eqref{24} and consider the ${\theta }{\wedge }{\theta ^{%
\overline{\gamma }}}$ terms to obtain
\begin{equation}
\overset{\centerdot }{{A}}_{\overline{\alpha }\overline{\gamma }}=-2(i\eta _{%
\overline{\alpha }\overline{\gamma }}+\eta {A}_{\overline{\alpha }\overline{%
\gamma }})-iE{_{\overline{\alpha }\overline{\gamma }}},_{0}.  
\label{26}
\end{equation}%
On the other hand, contracting (\ref{25}) with $Z_{\beta }$ and then
contracting with $h^{\beta \overline{\alpha }},$ computing modulo $\theta
^{\gamma }$ yields%
\begin{equation*}
\overset{\centerdot }{{\omega }}{_{\alpha }}^{\alpha }=i(A{_{\alpha }}^{%
\overline{\gamma }}E{_{\overline{\gamma }}}^{\alpha }+A{_{\overline{\alpha }}%
}^{\gamma }E{_{\gamma }}^{\overline{\alpha }}+\eta {_{\overline{\alpha }}}^{%
\overline{\alpha }})\theta -[(n+2)\eta _{\overline{\alpha }}+iE{_{\overline{%
\gamma }\overline{\alpha }}},^{\overline{\gamma }}]{\theta ^{\overline{%
\alpha }}}\text{ }\mod \theta ^{\gamma }.
\end{equation*}%
Since $\overset{\centerdot }{{\omega }}{_{\alpha }}^{\alpha }$ is pure
imaginary, we have
\begin{equation}
\begin{array}{lll}
\overset{\centerdot }{{\omega }}{_{\alpha }}^{\alpha } & = & i(A{_{\alpha }}%
^{\overline{\gamma }}E{_{\overline{\gamma }}}^{\alpha }+A{_{\overline{\alpha 
}}}^{\gamma }E{_{\gamma }}^{\overline{\alpha }}+\Delta _{b}\eta )\theta \\ 
&  & +[(n+2)\eta _{\alpha }-iE{_{\gamma \alpha }},^{\gamma }]{\theta
^{\alpha }-[(n+2)\eta _{\overline{\alpha }}+iE{_{\overline{\gamma }\overline{%
\alpha }}},^{\overline{\gamma }}]\theta ^{\overline{\alpha }}.}%
\end{array}
\label{27}
\end{equation}%
Differentiate the structure equation \eqref{eq:curvature_Tanaka_connection} with respect to $t$ and
consider only the $\theta ^{\rho }\wedge \theta ^{\bar{\sigma}}$ terms.
This gives%
\begin{equation}
\begin{array}{ccl}
\overset{\centerdot }{R}_{\rho \bar{\sigma}} & = & -(A{_{\alpha }}^{%
\overline{\gamma }}E{_{\overline{\gamma }}}^{\alpha }+A{_{\overline{\alpha }}%
}^{\gamma }E{_{\gamma }}^{\overline{\alpha }}+\Delta _{b}\eta )h_{\rho \bar{%
\sigma}}-2\eta R_{\rho \bar{\sigma}} \\ 
&  & -[(n+2)\eta _{\rho }-iE{_{\gamma \rho }},^{\gamma }],_{\overline{\sigma 
}}{-[(n+2)\eta _{\overline{\sigma }}+iE{_{\overline{\gamma }\overline{\sigma 
}}},^{\overline{\gamma }}]},_{\rho }%
\end{array}
\label{28}
\end{equation}%
After contracting with $h^{\rho \bar{\sigma}}$ we get%
\begin{equation}
\begin{array}{ccl}
\overset{\centerdot }{W} & = & i(E{_{\gamma \alpha }},^{\gamma \alpha }-{E{_{%
\overline{\gamma }\overline{\alpha }}},^{\overline{\gamma }\overline{\alpha }%
}})-n(A{_{\alpha }}^{\overline{\gamma }}E{_{\overline{\gamma }}}^{\alpha }+A{%
_{\overline{\alpha }}}^{\gamma }E{_{\gamma }}^{\overline{\alpha }}) \\ 
&  & -[2(n+1)\Delta _{b}\eta +2W\eta ] \\ 
& = & 2\re \left( iE{_{\gamma \alpha }},^{\gamma \alpha }-nA{_{\alpha }}%
^{\overline{\gamma }}E{_{\overline{\gamma }}}^{\alpha }\right)
-[2(n+1)\Delta _{b}\eta +2{W\eta ].}%
\end{array}
\label{eq:variationW}
\end{equation}

Recall that the transformation law of the connection under a change of pseudohermitian
structure was computed in \cite[Sec. 5]{l1}. 
Let $\hat{\theta}=e^{2f}\theta$ be another pseudohermitian structure. Then we can define an admissible coframe
by $\hat{\theta}^{\alpha }=e^{f}(\theta ^{\alpha }+2if^{\alpha }\theta )$.
With respect to this coframe, the connection $1$-form and the pseudohermitian torsion
are given by%
\begin{equation}
\label{30}
\begin{split}
\widehat{{\omega }}{_{\beta }}^{\alpha } &={\omega _{\beta }}^{\alpha
}+2(f_{\beta }\theta ^{\alpha }-f^{\alpha }\theta _{\beta })+\delta _{\beta
}^{\alpha }(f_{\gamma }\theta ^{\gamma }-f^{\gamma }\theta _{\gamma })
 \\
&\phantom{=}+i(f^{\alpha }{}_{\beta }+f_{\beta }{}^{\alpha }+4\delta _{\beta }^{\alpha
}f_{\gamma }f^{\gamma })\theta ,  
\end{split}
\end{equation}
and
\begin{equation}
\widehat{{A}}{_{\alpha \beta }=}e^{-2f}({A_{\alpha \beta }+2i}f_{\alpha
\beta }-4if_{\alpha }f_{\beta }),  \label{31}
\end{equation}%
respectively. 
Thus the Webster curvature transforms as
\begin{equation}
\widehat{W}=e^{-2f}(W-2(n+1)\Delta _{b}f-4n(n+1)f_{\gamma }f^{\gamma }).
\label{32}
\end{equation}%
Here covariant derivatives on the right side are taken with respect to the
pseudohermitian structure $\theta $ and an admissible coframe $\theta
^{\alpha }$. Note also that the dual frame of $\{\hat{\theta},\hat{\theta}%
^{\alpha },\hat{\theta}^{\overline{\alpha }}\}$ is given by $\{\widehat{T},%
\widehat{Z}_{\alpha },\widehat{Z}_{\overline{\alpha }}\}$, where

\begin{equation*}
\widehat{T}=e^{-2f}(T+2if^{\overline{\gamma }}Z_{\overline{\gamma }%
}-2if^{\gamma }Z_{\gamma }),\text{ \ }\widehat{Z}_{\alpha }=e^{-f}Z_{\alpha
}.
\end{equation*}%

\end{document}

%% file: repelling1.pdf_tex
\begingroup%
  \makeatletter%
  \providecommand\color[2][]{%
    \errmessage{(Inkscape) Color is used for the text in Inkscape, but the package 'color.sty' is not loaded}%
    \renewcommand\color[2][]{}%
  }%
  \providecommand\transparent[1]{%
    \errmessage{(Inkscape) Transparency is used (non-zero) for the text in Inkscape, but the package 'transparent.sty' is not loaded}%
    \renewcommand\transparent[1]{}%
  }%
  \providecommand\rotatebox[2]{#2}%
  \ifx\svgwidth\undefined%
    \setlength{\unitlength}{555bp}%
    \ifx\svgscale\undefined%
      \relax%
    \else%
      \setlength{\unitlength}{\unitlength * \real{\svgscale}}%
    \fi%
  \else%
    \setlength{\unitlength}{\svgwidth}%
  \fi%
  \global\let\svgwidth\undefined%
  \global\let\svgscale\undefined%
  \makeatother%
  \begin{picture}(1,0.65045045)%
    \put(0,0){\includegraphics[width=\unitlength]{repelling1.pdf}}%
  \end{picture}%
\endgroup%